\documentclass[12pt,reqno]{amsart}
\usepackage{amstext}
\usepackage{amsthm}
\usepackage{amsmath}
\usepackage{amssymb}
\usepackage{latexsym}
\usepackage{amsfonts}
\usepackage{graphicx}
\usepackage{texdraw}
\usepackage{graphpap}

\usepackage[pagebackref,hypertexnames=false, colorlinks, citecolor=red, linkcolor=red]{hyperref} %

\usepackage[backrefs]{amsrefs}
\allowdisplaybreaks

\begin{document}

\newcommand{\hilite}[1]{\colorbox{yellow}{$\displaystyle #1$}}
\let\underdot=\d 

\providecommand{\abs}[1]{\left\lvert#1\right\rvert} 
\providecommand{\acc}{\mathrm{acc}} 
\providecommand{\Aut}{\mathrm{Aut}} 
\providecommand{\C}{\mathbb{C}} 
\providecommand{\diag}{\mathrm{diag}} 
\providecommand{\F}{\mathbb{F}} 
\providecommand{\Gal}{\mathrm{Gal}} 
\providecommand{\GL}{\mathrm{GL}} 
\providecommand{\Ha}{\mathbb{H}} 
\providecommand{\Image}{\mathrm{Im}} 
\providecommand{\Inn}{\mathrm{Inn}} 
\providecommand{\Ker}{\mathrm{Ker}} 
\providecommand{\lcm}{\mathrm{lcm}} 
\providecommand{\N}{\mathbb{N}} 
\providecommand{\norm}[1]{\left\lVert#1\right\rVert} 
\providecommand{\Null}{\mathrm{Nullity}} 
\providecommand{\Q}{\mathbb{Q}} 
\providecommand{\R}{\mathbb{R}} 
\providecommand{\Rank}{\mathrm{Rank}} 
\providecommand{\SL}{\mathrm{SL}} 
\providecommand{\Span}{\mathrm{span}} 
\providecommand{\todo}{\hilite{\ensuremath{\backslash}todo}} 
\providecommand{\tr}{\mathrm{tr}} 
\providecommand{\Z}{\mathbb{Z}} 
\providecommand{\Sch}{\mathcal{S}} 
\providecommand{\S}{\bold{S}}
\providecommand{\Langle}{\left\langle}
\providecommand{\Rangle}{\right\rangle}
\providecommand{\lp}{\left(}
\providecommand{\rp}{\right)}
\newcommand{\sgn}{\operatorname{sgn}}
\providecommand{\Bn}{\mathbb{B}^n}
\providecommand{\Cd}{\mathbb{C}^d}
\providecommand{\Cn}{\mathbb{C}^n}
\providecommand{\fl}{\mathcal{L}}
\providecommand{\ft}{\mathcal{T}}
\providecommand{\dva}{dv_{\alpha}}
\providecommand{\lpa}{L_{\alpha}^{p}}
\providecommand{\apa}{A_{\alpha}^{p}}
\providecommand{\supp}{\textnormal{supp }}
\providecommand{\diam}{\textnormal{diam}}
\providecommand{\vr}{\varrho}
\providecommand{\K}[2]{K_{#1}^{(#2,\alpha)}}
\providecommand{\nk}[2]{k_{#1}^{(#2,\alpha)}}

\renewcommand{\qedsymbol}{\textsquare} 
\renewcommand{\d}[1]{\,d#1} 
\newcommand{\ip}[2]{\ensuremath{\left\langle#1,#2\right\rangle}}

 \newtheorem{thm}{Theorem}[section]
 \newtheorem{lem}[thm]{Lemma}
 \newtheorem{defn}[thm]{Definition}
  \newtheorem{cor}[thm]{Corollary}
 \newtheorem{prop}[thm]{Proposition}
 \theoremstyle{remark}
\newtheorem{exc}[thm]{Exercise}
 \newtheorem*{ex}{Example}
  \newtheorem{question}[thm]{Question}
\theoremstyle{definition}
\newtheorem{Def}{Definition}

\theoremstyle{remark}
\newtheorem{rem}[thm]{Remark}
\newtheorem*{rem*}{Remark}

\hyphenation{geo-me-tric}

\title[Essential Norm of Operators on Vector--Valued Bergman Space]
{The essential norm of operators on the bergman space of vector--valued
functions on the unit ball}

\author[R. Rahm]{Robert S. Rahm}
\address{Robert S. Rahm, School of Mathematics\\ Georgia Institute of Technology
\\ 686 Cherry Street\\ Atlanta, GA USA 30332-0160}
\email{rrahm3@math.gatech.edu}

\author[B. D. Wick]{Brett D. Wick$^\dagger$}
\address{Brett D. Wick, School of Mathematics\\ 
Georgia Institute of Technology\\ 686 Cherry Street
\\ Atlanta, GA USA 30332-0160}
\email{wick@math.gatech.edu}
\thanks{$\ddagger$ Research supported in part by National Science 
Foundation DMS grants \#0955432.}


\begin{abstract}
Let $A_{\alpha}^{p}(\Bn;\Cd)$ be the weighted Bergman space on the unit ball 
$\Bn$ of $\Cn$ of functions taking values in $\Cd$. For $1<p<\infty$ let 
$\mathcal{T}_{p,\alpha}$ be the algebra generated by finite sums of finite 
products of Toeplitz operators with bounded matrix--valued symbols (this is 
called the Toeplitz algebra in the case $d=1$). We show that every $S\in 
\mathcal{T}_{p,\alpha}$ can be approximated by localized operators. This will be
used to obtain several equivalent expressions for the essential norm of 
operators in $\mathcal{T}_{p,\alpha}$. We then use this to characterize compact 
operators in  $A_{\alpha}^{p}(\Bn;\Cd)$. The main result generalizes previous 
results and states that an operator in $A_{\alpha}^{p}(\Bn;\Cd)$
is compact if only if it is in $\mathcal{T}_{p,\alpha}$ and its Berezin 
transform vanishes on the boundary. 
\end{abstract}

\maketitle

\section{Introduction and Statement of Main Results}
\subsection{Definition of the Spaces $L_{\alpha}^{p}$ and $A_{\alpha}^{p}$}

Let \begin{math}\Bn\end{math} denote the open unit ball in 
\begin{math}\C^{n}\end{math}. Fix some \begin{math}d\in\N
\end{math}. If \begin{math}f\end{math} is a function defined on 
\begin{math}\Bn\end{math} taking values in \begin{math}\Cd\end{math}
(that is, \begin{math}f\end{math} is vector-valued), we say that \begin{math}f
\end{math} is measurable if \begin{math}z\mapsto \Langle f(z),e\Rangle_{\Cd}
\end{math} is measurable for every \begin{math}e\in\Cd \end{math}.
For \begin{math} \alpha > -1
\end{math}, let 
\begin{align*}
\dva (z) := c_{\alpha}(1-|z|^{2})^{\alpha}dV (z) 
\end{align*}
where $dV$ is volume measure on $\Bn$ and \begin{math}c_{\alpha} \end{math} is a
constant such that \begin{math}\int_{\Bn}\dva (z) = 1 \end{math}. 
For vectors in $\Cd$, let $\norm{\cdot}_{p}$ denote the $p$--norm on 
$\Cd$. That is, if $v=(v_1,\ldots,v_d)$ then
$\norm{v}_{p}:=\left(\sum_{i=1}^{d}\abs{v_i}^{p}\right)^{1/p}$.
Define 
\begin{math} L_{\alpha}^{p}(\Bn;\Cd)\end{math} to be the set of 
all measurable functions on \begin{math}\Bn\end{math} taking values in 
\begin{math}\Cd\end{math} such that 
\begin{align*}
\norm{f}_{L_{\alpha}^{p}(\Bn;\Cd)}^{p} 
:=\int_{\Bn}\norm{f(z)}_{p}^{p}dv_{\alpha}(z) < \infty.
\end{align*}
It should be noted that \begin{math}L_{\alpha}^{2}(\Bn;\Cd)\end{math} is a 
Hilbert Space with inner product:
\begin{align*}
\ip{f}{g}_{L_{\alpha}^{2}(\Bn;\Cd)} := \int_{\Bn}\ip{f(z)}{g(z)}_{\Cd}\dva (z).
\end{align*}

Similarly, a function \begin{math}f\end{math} is said to be 
holomorphic if \begin{math}z\mapsto\Langle f(z),e\Rangle_{\Cd} \end{math} is a 
holomorphic function for every \begin{math}e\in\C^{d}\end{math}. Since 
$\Cd$ is a finite dimensional space, this is equivilent to requiring that $f$
be holomorhpic in each component function. 
Define \begin{math}A_{\alpha}^{2}(\Bn;\Cd)\end{math} to be the 
set of holomorphic functions on $\Bn$ that are also in 
\begin{math}L_{\alpha}^{2}(\Bn;\Cd)\end{math}. Finally, let 
$\mathcal{L}(L_{\alpha}^{p}(\Bn;\Cd))$ denote the bounded linear 
operators on $L_{\alpha}^{p}$. Define $\mathcal{L}(A_{\alpha}^{p}(\Bn;\Cd))$
similarly.

\subsection{Background for the Scalar--Valued Case}
For the moment, let \begin{math}d=1\end{math}.
Recall the reproducing kernel:
\begin{align*}
K_{z}^{(\alpha)}(w)=K^{(\alpha)}(z,w):=\frac{1}{(1-\overline{z}w)^{n+1
+\alpha}}.
\end{align*} 

\noindent That is, if \begin{math}f\in A_{\alpha}^{2}(\Bn;\C)\end{math} 
there holds: 
\begin{align*}
f(z) = \ip{f}{K_{z}^{(\alpha)}}_{A_{\alpha}^{2}(\Bn;\C)}=
\int_{\Bn}\frac{f(w)}{(1-z\overline{w})^{n+1+\alpha}}\dva (w).
\end{align*}

\noindent Recall also the normalized reproducing kernels 
\begin{math}k_{z}^{(p,\alpha)}(w) =
\frac{(1-|z|^{2})^{\frac{n+1+\alpha}{q}}}
{(1-\overline{z}w)^{n+1+\alpha}}\end{math}, where 
$q$ is conjugate exponent to $p$. There holds that
$\norm{k_{z}^{(p,\alpha)}}_{A_{\alpha}^{p}(\Bn;\C)}\simeq 1$, where 
the implied constant is independent of $z$.

The reproducing kernels allow us to explicitly write the orthogonal 
projection from \begin{math}L_{2}^{\alpha}(\Bn;\C) \end{math} to 
\begin{math}A_{\alpha}^{2}(\Bn;\C) \end{math}:
\begin{align*}
(P_{\alpha}f)(z)= \ip{f}{K_{z}^{(\alpha)}}_{L_{\alpha}^{2}(\Bn;\C)}.
\end{align*}
Let \begin{math}\phi\in L^{\infty}(\Bn)\end{math}. The Toeplitz operator with 
symbol \begin{math}\varphi\end{math} is defined to be:
\begin{align*}
T_{\varphi}:=P_{\alpha}M_{\varphi}.
\end{align*}
Where \begin{math}M_{\phi}\end{math} is the multiplication operator. So, we have
that: \begin{math}(T_{\varphi}f)(z) = 
\ip{\varphi f}{K_{z}^{(\alpha)}}_{L_{\alpha}^{2}}\end{math}. If $T$ is an 
operator on $A_{\alpha}^{p}(\Bn;\C)$, the 
Berezin transform of \begin{math}T\end{math}, denoted 
\begin{math}\widetilde{T} \end{math} is a function on \begin{math}\Bn\end{math}
defined by the formula: \begin{math} \widetilde{T}(\lambda) = 
\ip{Tk_{\lambda}^{(p,\alpha)}}
{k_{\lambda}^{(q,\alpha)}}_{A_{\alpha}^{2}(\Bn)(\Bn;\C)} \end{math}.

\subsection{Generalization to Vector--Valued Case}
Now, we consider \begin{math}d\in\N\end{math} with $d>1$.
The preceding discussion can be carried over with a few modifications. 
First, the reproducing kernels remain the same, but the function 
\begin{math}f\end{math} is now \begin{math}\Cd\end{math}-valued and the 
integrals must be interpreted as vector--valued integrals 
(that is, integrate in 
each coordinate). To make this more precise, 
let $\{e_k\}_{k=1}^{d}$ be the standard orthonormal basis for $\Cd$. 
If 
\begin{math}f\end{math} is
a \begin{math}\Cd\end{math}-valued function on \begin{math}\Bn\end{math}, its 
integral is defined as:
\begin{align*}
\int_{\Bn}f(z)\dva (z) &:= \sum_{k=1}^{d} \left(\int_{\Bn} 
\ip{f(z)}{e_{k}}_{\Cd}\dva (z)\right)e_{k}.
\end{align*}

Let $L_{M^{d}}^{\infty}$ denote the set of $d\times d$ matrix--valued 
functions, $\varphi$,
such that the function 
\begin{math}z\mapsto\norm{\varphi(z)}_{\Cd\to\Cd}\end{math} 
is in \begin{math}L^{\infty}(\Bn;\C) \end{math}. Note that it is not 
particularly important which matrix norm is used, since 
\begin{math}\Cd\end{math} is finite dimensional and all norms are equivalent. 
The second change is that the symbols of Toeplitz operators are now 
matrix--valued functions in $L_{M^{d}}^{\infty}$.

Define $\mathcal{T}_{p,\alpha}$ to be the operator--norm topology 
closure of the set of finite sums of 
finite products of Toeplitz operators with $L_{M^{d}}^{\infty}$ symbols.

Finally, we change the way that we define the Berezin transform of an operator.
The Berezin transform will be a matrix--valued function, given by the following
relation (see also \cites{AE,K}):

\begin{align}
\ip{\widetilde{T}(z)e}{h}_{\Cd}=\ip{T(k_{z}^{(p,\alpha)} e)}
{k_{z}^{(q,\alpha)} h}_{A_{\alpha}^{2}}
\end{align}
for \begin{math}e,h\in\Cd\end{math}. (Again, $q$ is conjugate
exponent to $p$).

We are now ready to state the main theorem of the paper.
\begin{thm}Let \begin{math}1<p<\infty\end{math} and 
\begin{math}\alpha > -1\end{math} and \begin{math}S\in\mathcal{L}(\apa,\apa)
\end{math}. Then $S$ is compact if and only if \begin{math}S\in
\mathcal{T}_{p,\alpha}\end{math} and \begin{math}\lim_{\abs{z}\to 1}
\widetilde{S}(z)=0\end{math}.
\end{thm}

\subsection{Discussion of the Theorem.}By now, there are many results that 
relate the compactness of an operator to its
Berezin transform.
It seems that the first result in this direction is due to Axler and 
Zheng. In \cite{AZ} they prove that if 
$T\in \mathcal{L}(A_{0}^{2}(\mathbb{B};\C))$ 
can be written as a finite sum of finite products of Toeplitz operators,
then $T$ is compact if and only if its 
Berezin transform vanishes on the boundary of $\mathbb{B}$ (recall that 
$(A_{0}^{2}(\mathbb{B};\C))$ is the standard Bergman space on the 
unit ball in $\C$). There are several results generalizing this to 
larger classes of operators, more general domains, and weighted Bergman spaces.
See, for example \cites{RAI,eng,cs,mz}.

There are also several results along these lines for more general operators
than those that can be written as finite sums of finite products of 
Toeplitz operators. In \cite{ENG} Engli\v{s} proves that any compact
operator is in the operator--norm topology closure of the set of finite
sums of finite products of Toeplitz operators (this is called the 
Toeplitz algebra).
In \cite{Sua}, Su{\'a}rez proves that an operator, $T\in 
\mathcal{L}(A_{0}^{p}(\Bn;\C))$ is compact if and only if it is in the
Toeplitz algebra and its Berezin transform vanishes on $\partial\Bn$. This
was extended to the weighted Bergman spaces $A_{\alpha}^{p}(\Bn;\C)$
in \cite{MSW} by Su{\'a}rez, Mitkovski, and Wick and to
Bergman spaces on the polydisc and bounded symmetric domains by Mitkovski and 
Wick in \cite{mw1} and \cite{mw2}.

\section{Preliminaries}
We first fix notation that will last for the rest of the paper. The vectors 
$\{e_{i}\}_{i=1}^{d}$, etc. will denote the standard orthonormal basis 
vectors in $\Cd$. The letter $e$ will always denote a unit vector in $\Cd$. 
For vectors in $\Cd$, 
$\norm{\cdot}_{p}$ will denote the $l^{p}$ norm on $\Cd$. If 
$M$ is a $d\times d$ matrix, $\norm{M}$ will denote any convenient matrix norm.
Since all norms of matrices are equivalent in finite dimensions, the exact norm
used does not matter for our considerations. Additionally, 
$M_{(i,j)}$ will denote the $(i,j)$ entry of $M$ and $E_{(i,j)}$ will be the
matrix whose $(i,j)$ entry is $1$ and all other entries are $0$.
Finally, to lighten notation, fix an integer $d>1$, an integer $n\geq 1$ and a 
real $\alpha > -1$. Because of
this, we will usually suppress these constants in our notation. 

\subsection{Well-Known Results and Extensions to the Present Case}
We will discuss several well-known results about the standard Bergman Spaces, 
$A_{\alpha}^{p}(\Bn;\C)$ and state and prove their generalizations to the 
present vector-valued Bergman Spaces, $\apa$. 

Recall the automorphisms, $\phi_{z}$, of the ball that 
interchange $z$ and $0$. The automorphisms are used to define the following 
metrics:
\begin{align*}
\rho(z,w):=\abs{\phi_{z}(w)} \quad\textnormal{and}\quad 
\beta(z,w):=\frac{1}{2}\log\frac{1+\rho(z,w)}{1-\rho(z,w)}.
\end{align*}

These metrics are invariant under the maps $\phi_{z}$. Define $D(z,r)$ to be the
ball in the $\beta$ metric centered at $z$ with raduis $r$. Recall the following
identity:
\begin{align*}
 1-\abs{\phi_{z}(w)}^{2}=\frac{(1-\abs{z}^{2})(1-\abs{w}^{2})}
 {\abs{1-\overline{z}w}^{2}}.
\end{align*}

The following change of variables formula is \cite{Zhu}*{Prop 1.13}:
\begin{align}
\label{eqn:cov}
\int_{\Bn}f(w)\dva (w) = \int_{\Bn}(f\circ \phi_{z})(w) 
\abs{\nk{z}{2}(w)}^2 \dva (w). 
\end{align}
The following propositions appear in \cite{Zhu}.

\begin{prop}\label{prop:soh2.20}
If \begin{math}a\in \Bn\end{math} and \begin{math}z\in D(a,r)\end{math}, 
there exists a constant depending only on \begin{math}r\end{math} such that 
\begin{math}1-|a|^2 \simeq 1-|z|^{2} \simeq \abs{1-\ip{a}{z}} \end{math}.
\end{prop}

\begin{prop}\label{prop:soh2.24}
Suppose $r>0$, $p>0$, and $\alpha > -1$. Then there exists a constant $C>0$ such
that
\begin{align*}
\abs{f(z)}^{p} \leq \frac{C}{(1-|z|^{2})^{n+1+\alpha}}
\int_{D(z,r)}\abs{f(w)}^{p}dv_{\alpha}(w)
\end{align*}
for all holomorphic \begin{math}f:\Bn\to\C\end{math} and all 
\begin{math}z\in \Bn
\end{math}.
\end{prop}
\noindent The following vector-valued analogue will be used: 
\begin{prop}\label{prop:vvsoh2.24}
Let \begin{math}\lambda\in \Bn\end{math}. There exists a constant 
\begin{math}C>0\end{math} such that
\begin{align*}
\sup_{z\in D(\lambda,r)}\norm{f(z)}_{p}^{p} \leq
\frac{C}{(1-|\lambda|^{2})^{n+1+\alpha}}
\int_{D(\lambda,2r)}\norm{f(w)}_{p}^{p}dv_{\alpha}(w).
\end{align*}
\end{prop}

\begin{proof}
First note that \begin{math}\sup_{z\in D(\lambda,r)}\|f(z)\|_{p}^{p} = 
\sup_{z\in D(\lambda,r)}\sup_{\|e\|_{q}=1}\abs{\ip{e}{f(z)}}^{p}\end{math}. 
By definition, \begin{math}\langle e,f(z)\rangle_{\Cd}\end{math} is holomorphic 
for all \begin{math}e\in \C^{d}\end{math}. By Proposition \ref{prop:soh2.24} and
Proposition \ref{prop:soh2.20}, for 
\begin{math}\|e\|_{\Cd}=1\end{math} and \begin{math}z\in D(\lambda,r)\end{math} 
there holds:
\begin{align*}
|\langle e,f(z) \rangle_{\Cd} |^{p} &\leq \frac{C}{(1-|z|^{p})^{n+1+\alpha}}
\int_{D(z,r)}|
\langle e,f(w)\rangle_{\Cd}|^{p}dv_{\alpha}(w)
\\& \leq \frac{C}{(1-|z|^{2})^{n+1+\alpha}}\int_{D(z,r)}
   \|f(w)\|_{p}^{p}dv_{\alpha}(w)
\\& \simeq \frac{C}{(1-|\lambda|^{2})^{n+1+\alpha}}\int_{D(z,r)}
   \|f(w)\|_{p}^{p}dv_{\alpha}(w)
\\& \leq \frac{C}{(1-|\lambda|^{2})^{n+1+\alpha}}\int_{D(\lambda,2r)}
   \|f(w)\|_{p}^{p}dv_{\alpha}(w).
\end{align*}
Which completes the proof.
\end{proof}

The next lemma is in \cite{Zhu}:

\begin{lem}
\label{Growth}
For $z\in \Bn$, $s$ real and $t>-1$, let
\begin{align*}
F_{s,t}(z):=\int_{\Bn}\frac{(1-\abs{w}^2)^t}{\abs{1-\overline{w}z}^s}\,dv(w).
\end{align*}
Then $F_{s,t}$ is bounded if $s<n+1+t$ and grows as $(1-\abs{z}^2)^{n+1+t-s}$ 
when $\abs{z}\to 1$ if $s>n+1+t$.
\end{lem}

We now give several geometric decompositions of the ball. See \cite{Zhu} for
the proofs.

\begin{lem}
\label{StandardGeo}
Given $\varrho>0$, there is a family of Borel sets $D_m\subset\Bn$ and points 
$\{w_m\}_{m=1}^{\infty}$ such that
\begin{description}
\item[(i)] $D\left(w_m,\frac{\varrho}{4}\right)\subset 
D_m\subset D\left(w_m,\varrho\right)$ for all $m$;
\item[(ii)] $D_k\cap D_l=\emptyset$ if $k\neq l$;
\item[(iii)] $\bigcup_{m=1}^{\infty} D_m=\Bn$.
\end{description}
\end{lem}

\begin{prop}\label{prop:soh2.23}
There exists a positive integer \begin{math}N\end{math} such that for any 
\begin{math}0<r\leq 1\end{math} we can 
find a sequence \begin{math}\{a_{k}\}_{k=1}^{\infty}
\text{ in }\Bn\end{math} with the 
following properties:
\begin{description}
    \item[(i)]$\Bn=\cup_{k=1}^{\infty}D(a_{k},r)$
    \item[(ii)]The sets $D(a_{k},\frac{r}{4})$ are mutually disjoint.
    \item[(iii)] Each point $z\in B_{n}$ belongs to at most $N$ 
    of the sets $D(a_{k},4r)$.
\end{description}
\end{prop}

The following lemma appears in \cite{Sua}.
\begin{lem}\label{lem:mswLem2.6}
Let $\sigma>0$ and $k$ be a non-negative integer.  For each $0\leq i\leq k$ the
family of sets $\mathcal{F}_{i}=\{F_{i,j}: j\geq 1\}$  forms a covering of 
$\Bn$ such that
\begin{description}
\item[(i)] $F_{0,j_1}\cap F_{0,j_2}=\emptyset$ if $j_1\neq j_2$;
\item[(ii)] $F_{0,j}\subset F_{1,j}\subset\cdots\subset F_{k,j}$ for all $j$;
\item[(iii)] $\beta(F_{i,j}, F_{i+1,j}^c)\geq\sigma$ for all $0\leq i\leq k$ and
             $j\geq 1$;
\item[(iv)] every point of $\Bn$ belongs to no more than $N$ elements of 
             $\mathcal{F}_{i}$;
\item[(v)] $\diam_{\beta} \,F_{i,j}\leq C(k,\sigma)$ for all $i,j$.
\end{description} 
\end{lem}

\subsection{Matrix-Valued Measures and Their $L^{P}$ Spaces}
We will be concerned with matrix--valued measures, \begin{math}\mu\end{math}. 
Loosely speaking, a
matrix--valued measure is a matrix--valued function on a $\sigma$-algebra 
such that every entry of the matrix
is a complex measure. More precisely, a matrix--valued measured is a matrix 
valued--function, $\mu$, on a $\sigma$-algebra such that \begin{math}
\mu(\emptyset)=0
\end{math} and that satisfies countable additivity. 

The matrix--valued analogue of non--negative measures are measures such 
that \begin{math}\mu(E)\end{math} is a positive semi-definite (PSD) matrix for 
every Borel subset of $\Bn$. For 
every matrix--valued measure, $\mu$, we associate to the matrix its trace 
measure \begin{math}\tau_{\mu}:=\sum_{i=1}^{d}\mu_{(i,i)}\end{math}. Since the 
trace of a matrix is the sum of its eigenvalues, and since a PSD matrix has no 
negative eigenvalues, $\tau_{\mu}$ is a non--negative scalar--valued measure
when $\mu$ is a PSD matrix--valued measure. Also, if the trace of a PSD matrix
is zero, the matrix is the zero matrix. This implies that  
\begin{math}\mu_{(i,j)}\ll\tau_{\mu}\end{math} and so the 
Lebesgue-Radon-Nikodym derivative, \begin{math}\frac{d\mu_{(i,j)}}
{d\tau_{\mu}}\end{math} is well defined $\tau_{\mu}$-a.e.. Let $M_{\mu}(z)$ 
denote the matrix whose $(i,j)$ entry is \begin{math}\frac{d\mu_{(i,j)}(z)}
{d\tau_{\mu}}\end{math}. The following decomposition of the PSD matrix-valued 
measure $\mu$ holds $\tau_{\mu}$-a.e.:
\begin{align*}
 d\mu(z)=M_{\mu}(z)d\tau_{\mu}(z).
\end{align*}

If \begin{math}
A\end{math} is a PSD matrix, and \begin{math}p\geq 1\end{math}, we can define a 
\begin{math}p^{\text{th}}\end{math}-power of 
\begin{math}A\end{math} by the following: We have that \begin{math}A=U^{*}
\Lambda U\end{math} where \begin{math}U\text{ is unitary and }\Lambda\end{math}
is diagonal with the eigenvalues of \begin{math}A\end{math} on the diagonal. 
Then we define \begin{math}A^{p}=U^{*}\Lambda^{p}U\end{math}. Using this 
definition, every PSD matrix \begin{math}A\end{math} has a unique PSD 
\begin{math}p^{\text{th}}\end{math}-root \begin{math}B\end{math} given by the 
folrmula: \begin{math}B=U^{*}\Lambda^{1/p}U\end{math}.
\noindent Consider the following preliminary definition:
\begin{Def}
Let \begin{math}L^{p}_{*}(\Bn,\C^{d};\mu)\end{math} be the
set of all \begin{math}\C^{d}\end{math}-valued functions that satisfy:
\begin{align*}
\|f\|_{L_{*}^{p}(\Bn,\C^{d};\mu)}^{p}:=\int_{\Bn}
\|M_{\mu}^{1/p}(w)f(w)\|_{p}^{p}d\tau_\mu(w) < \infty.
\end{align*}
\end{Def}
That \begin{math}\norm{f}_{L^{p}(\Bn,\C^{d};\mu)}\end{math} is a 
seminorm is an easy consequence of the fact that \begin{math}\|\cdot\|_{p}
\end{math} is a norm. However, it is not a norm because if \begin{math}
f(z)\in \ker M(z)\end{math} \begin{math}\tau_\mu\end{math}-a.e. then 
\begin{math}\|f\|_{L^{p}(B_{n},\C^{d};\mu)}=0\end{math}. We therefore 
define the following equivalence relation: \begin{math}f \sim_{M} g\end{math} 
if and only if \begin{math}M(z)f(z)=M(z)g(z)\end{math}  
\begin{math}\tau_\mu\end{math}-a.e. And we define \begin{math}
L^{p}(\Bn,\C^{d};\mu) = L^{p}_{*}(\Bn;\C^{d},\mu) / \sim_{M}\end{math}. We
similarly define $A^{p}(\Bn;\Cd,\mu)$ to be the set of holomorphic functions
that are also in $L^{p}(\Bn,\C^{d};\mu)$.

In the special case \begin{math}p=2\end{math},
\begin{math}L^{2}(\Bn,\C^{d};\mu)\end{math} is a Hilbert Space with inner 
product:
\begin{align*}
\Langle f,g\Rangle_{L^{2}(\Bn;\C^{d},\mu)}&=\int_{\Bn}
\Langle M_\mu(z)f(z),g(z)\Rangle_{\C^{d}}d\tau_\mu(z)
\\&= \int_{\Bn}\Langle d\mu(z)f(z),g(z)\Rangle_{\C^{d}}
\end{align*}

There is also the expected H{\"o}lder inequality:  
\begin{prop}\label{prop:holder}
Let \begin{math}\mu\end{math} be a PSD matrix measure on \begin{math}\Bn
\end{math}, $1<p<\infty$ and $q$ conjugate exponent to $p$. Then:
\begin{align*}
\abs{\ip{f}{g}_{L^{2}(\Bn,\C^{d};\mu)}} \leq 
\norm{f}_{L^{p}(\Bn,\C^{d};\mu)}
\norm{g}_{L^{q}(\Bn,\C^{d};\mu)}.
\end{align*}
\end{prop}
\begin{proof}The proof is a simple computation that uses linear algebra and
the usual H{\"o}lder's inequality. Indeed, 
\begin{align*}
\abs{\Langle f,g\Rangle_{L^{2}(\Bn;\C^{d},\mu)}} &=
    \abs{\int_{\Bn}\langle M_{\mu}(w) f(w),g(w)\rangle d\tau_\mu(w)}
\\&\leq \int_{\Bn}|\langle M_{\mu}^{1/p}(w)f(w),M_{\mu}^{1/q}(w)g(w)
   \rangle_{\C^{d}}|d\tau_\mu(w)
\\&\leq \int_{\Bn} \|M_{\mu}^{1/p}(w)f(w)\|_{\C^{d}}
   \|M_{\mu}^{1/q}(w)g(w)\|_{\C^{d}}d\tau_\mu(w)
\\&\leq \left(\int_{\Bn}\|M_{\mu}^{1/p}(w)f(w)
   \|_{\C^{d}}^{p}d\tau_\mu(w)\right)^{1/p}
   \left(\int_{\Bn}\|M_{\mu}^{1/q}(w)g(w)\|_{\C^{d}}^{q}
   d\tau_\mu(w)\right)^{1/q}
\\&=\norm{f}_{L^{p}(\Bn,\C^{d};\mu)} \norm{g}_{L^{q}(\Bn,\C^{d};\mu)}.
\end{align*}
\end{proof}
\subsection{Matrix-Valued Carleson Measures}
We will need to have a concept of matrix-valued Carleson measures.
A PSD matrix--valued measure \begin{math}\mu\end{math} on \begin{math}
\Bn\end{math} is a \textit{Carleson matrix--valued measure} for \begin{math}\apa
\end{math} if there is a constant 
\begin{math}C_{p}\end{math}, independent of \begin{math}f\end{math}, such that
\begin{align}\label{eqn:carDef}
\left(\int_{\Bn}\norm{M_{\mu}^{1/p}(z)f(z)}_{p}^{p}d\tau_\mu(z)\right )^{1/p} 
\leq C_p
\left(\int_{\Bn}\norm{f(z)}_{p}^{p} dv_{\alpha}(z)\right)^{1/p}.
\end{align}
The best constant for which (\ref{eqn:carDef}) holds will be denoted by 
\begin{math}
\norm{\iota_{(p,d)}}\end{math}. In the case that \begin{math}p=2\end{math}, the 
preceding inequality can be 
written in the following manner: 
\begin{align*}\label{eqn:ccmDef}
\left(\int_{\Bn}\ip{d\mu(z)f(z)}{f(z)}_{\C^{d}}\right )^{1/2} 
&\leq C_2\left(\int_{\Bn}\ip{dv_{\alpha}(z)f(z)}{f(z)}_{\C^{d}} \right)^{1/2}
\\&=C_2\left(\int_{\Bn}\ip{f(z)}{f(z)}_{\C^{d}}dv_{\alpha}(z) \right)^{1/2}.
\end{align*}

We now to state and give a proof of a Carleson Embedding Theorem for 
matrix--valued measures. We start by 
defining a generalization of Toeplitz operators. For $\mu$ a matrix--valued 
measure, define:
\begin{align*}
 T_{\mu}f(z):=\int_{\Bn}\frac{d\mu(w)f(w)}{(1-\overline{w}z)^{n+1+\alpha}}.
\end{align*}

\begin{lem}[Carleson Embedding Theorem] \label{lem:car}
For a PSD matrix--valued measure, $\mu$, the following quantities are 
equivalent:
\begin{description}
\item[(i)] $\|\mu\|_{RKM}:=
    \sup_{e\in\C^{d}, \|e\|_{2}=1}\sup_{\lambda\in \Bn}
    \int_{\Bn}|k_{\lambda}^{(2,\alpha)}(z)|^{2}\ip{d\mu(z)e}{e}_{\C^{d}}$;
\item[(ii)] $\|\iota_{(p,d)}\|^{p}:=
    \inf\left\{C:\int_{\Bn}\|M_{\mu}^{1/p}(z)f(z)\|_{\C^{d}}^{p} \leq C 
    \norm{f}_{A_\alpha^{p}}^{p}\right\}$;
\item[(iii)] $\|\mu\|_{GEO}:=
    \sup_{\lambda\in \Bn}\int_{D(\lambda,r)}\|M_{\mu}(z)\|_{\Cd}d\tau_\mu(z) 
    (1-|\lambda|^{2})^{-(n+1+\alpha)}$;
\item[(iv)] $B=
    \sup_{\lambda_{k}\in \Bn}\int_{D(\lambda_{k},r)}\|M_{\mu}(z)\|_{\Cd}
    d\tau_\mu(z) 
    (1-|\lambda_{k}|^{2})^{-(n+1+\alpha)}$
    where $\{\lambda_{k}\}_{k=1}^{\infty}$ is the sequence from     Proposition 
    \ref{prop:soh2.23};
\item[(v)] $\|T_{\mu}\|_{\mathcal{L}(A_{\alpha}^{p})}$.
\end{description}
\end{lem}

\begin{lem}\label{lem:equivTr}
Let \begin{math}A\in M^{d\times d}\end{math} be PSD, let 
\begin{math}\|\cdot\|\end{math} be any matrix norm 
(see Section 5.6 of \cite{HJ}), and let \begin{math}\lambda_{1}(A)\end{math} 
denote the largest
eigenvalue of $A$. Then there holds
\begin{math}\tr (A) \simeq \norm{A} \simeq \lambda_{1}(A)\end{math} with implied
constant depending only 
on \begin{math}d\end{math}. 
\end{lem}
\begin{proof}
Recall that all norms on \begin{math}M^{d\times d}\end{math} are equivalent
with constants depending only on 
\begin{math}d\end{math}. We therefore need only show that \begin{math}
\tr(A)\simeq \|A\|_{F}\end{math} where
\begin{math}\|A\|_{F} = \sqrt{\tr(A^*A)}\end{math} (i.e., it is the Frobenius 
Norm or Hilbert-Schmidt Norm). 
Let \begin{math}\{\lambda_{i}\}_{i=1}^{d}\end{math} be the eigenvalues of 
\begin{math}A\end{math} arranged 
in decreasing order and note that
\begin{math}A^*A=A^2\end{math}. Then
\begin{math}
\tr(A^*A) = \sum_{i=1}^{d}\lambda_{i}^{2} 
= (\lambda_{1}^{2}) \sum_{i=1}^{d}(\lambda_{i}/\lambda_{1})^{2}
\leq d\lambda_{1}^{2} 
\leq d\left(\sum_{i=1}^{d}\lambda_{i}\right)^{2} 
= d\tr(A)^{2}
\end{math}.
Also,  
\begin{math}
\tr(A)^{2} = \left(\sum_{i=1}^{n}\lambda_{i}\right)^{2}
\leq 2^{n-1}\sum_{i=1}^{n} \lambda_{i}^{2}
= 2^{n-1}\tr(A^*A)
\end{math}.
Finally, \begin{math}\lambda_{1}(A)\leq\tr(A)\leq d\lambda_{1}(A)\end{math}. 
\end{proof}

The following is used in the next two lemmas. 
\begin{lem}\label{lem:trequiv}
\begin{math}
\sup_{e\in\Cd,\|e\|_{2}=1}\sup_{\lambda \in \Bn}\int_{\Bn}
|k_{\lambda}^{(2,\alpha)}(z)|^{2}\langle 
d\mu(z)e,e\rangle \simeq
\sup_{\lambda\in \Bn}\int_{\Bn}|k_{\lambda}^{(2,\alpha)}(z)|^{2}
d\tau_{\mu}(z)\end{math}.
\end{lem}
\begin{proof}
The proof is a simple calculation that uses Lemma \ref{lem:equivTr}. Indeed, 
\begin{align*}
\sup_{e\in\Cd,\norm{e}_{\Cd}=1}\sup_{\lambda \in \Bn}\int_{\Bn}
    \abs{k_{\lambda}^{(2,\alpha)}(z)}^{2}
    \ip{\d\mu(z)e}{e}_{\Cd} 
    &=\sup_{e\in\Cd,\norm{e}_{\Cd}=1}
    \sup_{\lambda \in \Bn}\int_{\Bn}
    \abs{k_{\lambda}^{(2,\alpha)}(z)}^{2}
    \ip{M_{\mu}(z)e}{e}_{\Cd}d\tau_{\mu}(z)
\\&\leq \sup_{\lambda \in \Bn}
    \int_{\Bn}\abs{k_{\lambda}^{(2,\alpha)}(z)}^{2} 
    \norm{M_{\mu}(z)}_{\Cd}d\tau_{\mu}(z) 
\\&\simeq \sup_{\lambda \in \Bn}
    \int_{\Bn}\abs{k_{\lambda}^{(2,\alpha)}(z)}^{2} 
    \sum_{i=1}^{d}\ip{M_{\mu}(z)e_{i}}{e_{i}}_{\Cd}d\tau_{\mu}(z)
\\&\simeq \sup_{\lambda \in \Bn}
    \int_{\Bn}\abs{k_{\lambda}^{(2,\alpha)}(z)}^{2} \sum_{i=1}^{d}
    \ip{d\mu(z)e_{i}}{e_{i}}_{\Cd}
\\&\simeq \sup_{\lambda \in \Bn}
    \int_{\Bn}\abs{k_{\lambda}^{(2,\alpha)}(z)}^{2} d\tau_{\mu}(z).
\end{align*}
This gives one of the required inequalities. For the next inequality there 
holds: 
\begin{align*}
\sup_{\lambda\in \Bn}\int_{\Bn}|k_{\lambda}^{(2,\alpha)}(z)|^{2} d\tau_{\mu}(z) 
&=\sup_{\lambda\in \Bn}\int_{\Bn}|k_{\lambda}^{(2,\alpha)}(z)|^{2} 
    \sum_{j=1}^{d}\langle d\mu(z) e_{j},e_{j} \rangle
\\&=\int_{\Bn}|k_{\lambda}^{(2,\alpha)}(z)|^{2} \sum_{j=1}^{d}
    \langle M_{\mu}(z) e_{j},e_{j} \rangle d\tau_\mu(z)
\\&=\sum_{j=1}^{d}\int_{\Bn}|k_{\lambda}^{(2,\alpha)}(z)|^{2} \langle d\mu(z) 
    e_{j},e_{j} \rangle
\\&\leq d\sup_{\|e\|_{2}=1}\sup_{\lambda \in \Bn}\int_{\Bn}
    |k_{\lambda}^{(2,\alpha)}(z)|^{2}\langle d\mu(z)e,e\rangle.
\end{align*}
This completes the proof. 
\end{proof}

\begin{rem}
 Note that the lemma was stated using the function 
 \begin{math}\abs{\nk{\lambda}{2}}\end{math} (because this is what will be
 needed), but it is true for any non--negative function. 
\end{rem}

We are now ready to prove Lemma \ref{lem:car}. (The proof is simply an 
appropriate adaptation of the proofs given in, for example, 
\cites{Zhu,MSW,zhu1}).
\begin{proof}\begin{math}\|\mu\|_{GEO}\simeq \|\mu\|_{RKM}\end{math}.\\
We will use Proposition \ref{lem:equivTr}, Proposition \ref{prop:soh2.20}, and 
Lemma \ref{lem:trequiv}. Then,
\begin{align*}
\sup_{\lambda \in \Bn} 
    \frac{\int_{D(\lambda,r)}\|M_{\mu}(z)\|d\tau_\mu(z)}{
    (1-|\lambda|^{2})^{n+1+\alpha}}
&= \sup_{\lambda \in \Bn}\int_{D(\lambda,r)}
    \frac{(1-|\lambda|^{2})^{n+1+\alpha}}
    {(1-|\lambda|^{2})^{2(n+1+\alpha)}}\|M_{\mu}(z)\|d\tau_\mu(z)
\\&\simeq \sup_{\lambda \in \Bn}\int_{D(\lambda,r)}
    \frac{(1-|\lambda|^{2})^{n+1+\alpha}}
    {\abs{1-\lambda\overline{z}}^{2(n+1+\alpha)}}
    \sum_{k=1}^{d}\Langle M_{\mu}(z)e_{k},e_{k}\Rangle_{\C^{d}}d\tau_\mu(z)
\\&= \sup_{\lambda \in \Bn}\int_{D(\lambda,r)}
    |k_{\lambda}^{(2,\alpha)}(z)|^{2}\tr(d\mu(z))
\\&\simeq \sup_{\|e\|_{2}=1}\sup_{\lambda \in \Bn}\int_{\Bn}
    |k_{\lambda}^{(2,\alpha)}(z)|^{2}\langle d\mu(z)e,e\rangle.
\end{align*}
\end{proof}

\begin{proof}\begin{math}\|T_\mu\|_{\mathcal{L}(A_{\alpha}^{p})}\lesssim \|\iota_{(p,d)}\|^{p}
\end{math}. Let $f,g\in H^{\infty}(\Bn;\Cd)$ ($H^{\infty}(\Bn;\Cd)$ is simply
the space of bounded holomorphic $\Cd$--valued functions on $\Bn$).
Then by Fubini's Theorem and 
H{\"o}lder's Inequality 
(Proposition \ref{prop:holder}):
\begin{align*}
\left |\Langle T_{\mu}f,g\Rangle_{A_{\alpha}^{2}}\right | 
&=\left|\int_{\Bn}\Langle\int_{\Bn}\frac{d\mu(w)f(w)}
    {(1-\overline{w}z)^{n+1+\alpha}},g(z)\Rangle_{\C^{d}}
    dv_{\alpha}(z)\right|
\\&=\abs{\int_{\Bn}\ip{d\mu(w)f(w)}
    {\int_{\Bn}\frac{g(z)}{(1-w\overline{z})^{n+1+\alpha}}dv_{\alpha}(z)}_{\Cd}}
\\&\leq \int_{\Bn}\left|\Langle M_{\mu}(w)f(w),g(w)\Rangle_{\C^{d}}\right|d\tau_\mu(w)
\\&\leq \|f\|_{L^{p}(\Bn;\C^{d},\mu)}\|g\|_{L^{q}(\Bn;\C^{d},\mu)}
\\&\leq \|\iota_{(p,d)}\|\|\iota_{(q,d)}
    \|\|g\|_{A_{\alpha}^{2}}\|f\|_{A_{\alpha}^{2}}.
\end{align*}
\end{proof}

\begin{proof}\begin{math}B \lesssim \|\mu\|_{GEO}\end{math}.
This is immediate from the definitions. 
\end{proof}

\begin{proof}\begin{math}\|\iota_{(p,d)}\|^{p} \lesssim  B\end{math}.
Let \begin{math}\{a_{k}\}_{k=1}^{\infty}\end{math} be the sequence from
Proposition 
\ref{prop:soh2.23}.
So, there holds 
\begin{align*}
\int_{D(\lambda_{k},r)}\|M_{\mu}(z)\|d\tau_\mu(z) 
(1-|\lambda_{k}|^{2})^{-(n+1+\alpha)}
\leq B
\end{align*} 
for all \begin{math}k\end{math}. 
Let \begin{math}f\end{math} be holomorphic and $D_k=D(\lambda_{k},2r)$.
\begin{align*}
\int_{\Bn}\norm{M_{\mu}^{1/p}(z)f(z)}_{p}^{p}
d\tau_{\mu}(z)
&\leq \sum_{k=1}^{\infty}\int_{D_k}
    \|f(z)\|_{p}^{p}\|M^{1/p}_{\mu}(z)
    \|^{p}d\tau_\mu(z)
\\&\simeq \sum_{k=1}^{\infty}\int_{D_k}\|f(z)\|_{p}^{p}
    \tr (M^{1/p}_{\mu}(z))^{p}d\tau_\mu(z)
\\&\simeq \sum_{k=1}^{\infty}\int_{D_k}\|f(z)\|_{p}^{p}
    \|M_{\mu}(z)\|d\tau_\mu(z)
\\&\leq \sum_{k=1}^{\infty}\sup_{z\in D_k}
    \|f(z)\|_{p}^{p}\int_{D_k}
    \|M_{\mu}(z)\|d\tau_\mu(z)
\\&\lesssim \sum_{k=1}^{\infty}\frac{C}{(1-|\lambda_{k}|^{2})^{n+1+\alpha}}
    \int_{D_k}\|f(w)\|_{p}^{p}dv_{\alpha}(w)
    \int_{D(\lambda_{k},r)}\|M_{\mu}(z)\|d\tau_\mu(z)
\\&=\sum_{k=1}^{\infty}C
    \int_{D_k}\|f(w)\|_{p}^{p}dv_{\alpha}(w)
    \int_{D_k}\frac{\|M_{\mu}(z)\|}
    {(1-|\lambda_{k}|^{2})^{(n+1+\alpha)}}d\tau_\mu(z)
\\&\leq\sum_{k=1}^{\infty}C B\int_{D_k}
    \|f(w)\|_{p}^{p} dv_{\alpha}(w)
\\&\leq CBN\|f\|_{A_{\alpha}^{p}}^{P}.
\end{align*}
Above we use the estimate from Proposition \ref{prop:vvsoh2.24} and the
last inequality is due to the fact that each \begin{math}z\in 
\Bn\end{math} belongs to at most 
\begin{math}N\end{math} of the sets \begin{math}D(\lambda_{k},2r)\end{math}.
\end{proof}

\begin{proof}\begin{math}\|\mu\|_{RKM} \lesssim 
\|T_{\mu}\|_{\mathcal{L}(A_{\alpha}^{p})}\end{math}.
Assume that \begin{math}T_{\mu}\in\mathcal{L}(A_{\alpha}^{p})\end{math}. Then

\begin{align*}
\Langle T_{\mu}(k_{\lambda}^{(p,\alpha)}e),k_{\lambda}^{(q,\alpha)}e
\Rangle_{A_{\alpha}^{2}}
&= \int_{\Bn} \Langle T_{\mu}(k_{\lambda}^{(p,\alpha)}e)(z),
k_{\lambda}^{(q,\alpha)}(z)e\Rangle_{\C^{d}}dv_{\alpha}(z)
\\&=\int_{\Bn}\int_{\Bn}\Langle
\frac{d\mu(w)(1-|\lambda|^{2})^{\frac{n+1+\alpha}{q}}}
{\lp(1-\overline{w}z)(1-\overline{\lambda}w)\rp^{n+1+\alpha}}e,
\frac{(1-|\lambda|^{2})^{\frac{n+1+\alpha}{p}}}
{(1-\overline{\lambda}z)^{n+1+\alpha}}e\Rangle_{\C^{d}}dv_{\alpha}(z)
\\&=\int_{\Bn}\int_{\Bn}\Langle
\frac{d\mu(w)(1-|\lambda|^{2})^{(n+1+\alpha)}}
{(1-\overline{\lambda}w)^{n+1+\alpha}}e,
\frac{K_{\lambda}^{\alpha}(z)}
{(1-\overline{z}w)^{n+1+\alpha}}e
\Rangle_{\C^{d}}dv_{\alpha}(z)
\\&=\int_{\Bn}\Langle\frac{d\mu(w)(1-|\lambda|^{2})^{n+1+\alpha}}
{(1-\overline{\lambda}w)^{n+1+\alpha}}e,
\frac{1}{(1-\overline{\lambda}w)^{n+1+\alpha}}e\Rangle_{\Cd}
\\&=\int_{\Bn}\Langle \frac{d\mu(w)(1-|\lambda|^{2})^{n+1+\alpha}}
{|1-\overline{\lambda}w|^{2(n+1+\alpha)}}e,e
\Rangle_{\C^{d}}
\\&=\int_{\Bn}|k_{\lambda}^{(2,\alpha)}(w)|^{2}\Langle 
d\mu(w)e,e\Rangle_{\C^{d}}.
\end{align*}

This computation implies:
\begin{align*}
\sup_{e\in\C^{d}, \|e\|_{2}=1}\sup_{\lambda\in \Bn} 
\int_{\Bn}|k_{\lambda}^{(2,\alpha)}(z)|^{2}\Langle d\mu(z)
e,e\Rangle_{\C^{d}} 
&=\sup_{e\in\C^{d}, \|e\|_{2}=1}\sup_{\lambda\in \Bn}
\Langle T_{\mu}(k_{\lambda}^{(p,\alpha)}e),
k_{\lambda}^{(q,\alpha)}e\Rangle_{A^{2}}
\\&\leq\sup_{\lambda\in \Bn} \|T_{\mu}\|_{A^{p}\to A^{p}}
\|k_{\lambda}^{(p,\alpha)}\|_{A_{\alpha}^{p}}
\|k_{\lambda}^{(q,\alpha)}\|_{A_{\alpha}^{q}}
\\&\simeq  \|T_{\mu}\|_{A^{p}\to A^{p}}.
\end{align*}
\end{proof}

\begin{proof}\begin{math}\|\mu\|_{RKM} \lesssim \|\iota_{(p,d)}\|^{p}\end{math}.
From the inequalities we have already proven, we have \begin{math} \|\mu\|_{RKM}
\lesssim \|T_{\mu}\|_{\mathcal{L}(A_{\alpha}^{p})} \lesssim 
\|\iota_{(p,d)}\|^{p}\end{math}.
\end{proof}

To state the following corollary, we first define the scalar 
total variation, denoted $\abs{\mu}$, of a matrix-valued measure, $\mu$. Let 
\begin{math}\abs{\mu}:=\sum_{i=1}^{d}
\sum_{j=1}^{d}\abs{\mu_{(i,j)}}\end{math},
where \begin{math}\abs{\mu_{(i,j)}}\end{math} is the total variation of the 
measure $\mu_{(i,j)}$. In the case that $\mu$ is a PSD matrix--valued
measure, there holds: \begin{math}d\abs{\mu}(z)= 
\sum_{i,j}d\abs{\mu_{i,j}}(z)=
\sum_{i,j}\abs{\frac{d\mu_{i,j}}{d\tau_{\mu}}(z)}d\tau_{\mu}(z)=
\sum_{i,j}\abs{(M_{\mu})_{i,j}(z)}d\tau_{\mu}(z)
\simeq\norm{M_{\mu}(z)}d\tau_{\mu}(z)\end{math}.

To emphasize, the total variation of a matrix--valued measure is a positive
\textit{scalar--valued} measure. This differs from the definition in, for 
example, \cite{RR} in which the total variation of a matrix--valued measure is
defined to be a PSD matrix--valued measure. But our definition is not 
with out precedent. For example, in 
\cite{DU}, the authors define the total variation of a 
\textit{vector--valued} measure to be a positive scalar--valued 
measure, though their definition is different from ours.
Even though our definition of the total variation
of a matrix--valued measure is different than the one appearing in,
for example \cite{RR} and \cite{DU}, this is nonetheless a reasonable 
definition: If
$\nu_1$ is a complex scalar measure and $\nu_2$ is a positive measure such
that $\nu_1 \ll \nu_2$, and if $d\nu_1=fd\nu_2$ then the total variation
of $\nu_1$ is defined by $d\abs{\nu_1}=\abs{f}d\nu_2$. So, in the case that 
$\mu$ is a PSD matrix--valued
measure, we are saying that $d\abs{\mu}=\norm{M_{\mu}}d\tau_{\mu}$.

\begin{cor} \label{cor:car1}
Let \begin{math}|\mu|\end{math} be the total variation of the PSD matrix--valued
measure \begin{math}\mu\end{math}.
The following quantities are equivalent. 
\begin{description}
\item[(i)] $\|\mu\|_{RKM}:=
    \sup_{e\in\C^{d}, \|e\|_{2}=1}\sup_{\lambda\in \Bn}
    \int_{\Bn}|k_{\lambda}^{(2,\alpha)}(z)|^{2}
    \Langle d\mu(z) e,e\Rangle_{\C^{d}}$;
\item[(ii)] $\|\iota_{(p,d)}\|^{p}:=
    \inf\left\{C:\int_{\Bn}\|M_{\mu}^{1/p}(z)f(z)\|_{p}^{p}d\tau_{\mu}(z) \leq 
    C\|f\|_{A_{\alpha}^{p}}^{p}\right\}$;
\item[(iii)] $\|\mu\|_{GEO}:=
    \sup_{\lambda\in \Bn}\int_{D(\lambda,r)}\|M_{\mu}(z)\|d\tau_\mu(z) 
    (1-|\lambda|^{2})^{-(n+1+\alpha)}$;
\item[(iv)] $\|T_{\mu}\|_{\mathcal{L}(A_{\alpha}^{p})}$;
\item[(v)] $\||\mu|\|_{RKM}=\sup_{\lambda\in \Bn},
    \int_{\Bn}|k_{\lambda}^{(2,\alpha)}(z)|^{2}d|\mu|(z)$;
\item[(vi)] $\|\kappa_{(p,d)}\|^{p}=\inf\left\{C:\int_{\Bn}
    \|f(z)\|_{p}^{p}d|\mu|(z) \leq C \|f\|_{A_{\alpha}^{p}}^p\right\}$;
\item[(vii)] $\||\mu |\|_{GEO}=\sup_{\lambda\in \Bn}
    \int_{D(\lambda,r)}d|\mu|(z) (1-|\lambda|^{2})^{-(n+1+\alpha)}$;
\item[(viii)] $\|T_{|\mu|}\|_{\mathcal{L}(A_{\alpha}^{p})}$.
\end{description}
\end{cor}

\begin{proof}
The equivalence between (i)--(iv) was proven in Lemma 
\ref{lem:car}, and the 
equivalence of (v)-(viii) is well--known
(see for example \cite{MSW} or \cite{zhu1}). 
To prove the current theorem, we only
need to ``connect'' the two sets of equivalencies. But this is easy since 
the quantities defined in (iii) and (vii) are equivalent.  
\end{proof}

\begin{cor}\label{cor:eleCar}
If $\mu$ is a Carleson matrix-valued measure or if $|\mu|$ is $\apa$-Carleson, 
then the variation of every entry of $\mu$ is Carleson. 
\end{cor}
\begin{proof}
There holds:
\begin{align*}
\int_{\Bn}\|f(z)\|_{p}^{p}d|\mu_{(i,j)}|(z) &= 
\int_{\Bn}\|f(z)\|_{p}^{p}|(M_{\mu})_{(i,j)}(z)|d\tau_\mu(z) 
\\&\leq\sum_{i=1}^{d}\sum_{j=1}^{d}
    \int_{\Bn}\|f(z)\|_p^{p}|M_{(i,j)}(z)|d\tau_\mu(z).
\\&\simeq \int_{\Bn}\|f(z)\|_{p}^{p}\|M_{\mu}(z)\|d\tau_\mu(z).
\end{align*} 
Using Corollary 
\ref{cor:car1}, 
\begin{align*}
\int_{\Bn}\|f(z)\|_{p}^{p}\|M_{\mu}(z)\|d\tau_\mu(z) &\simeq
\int_{\Bn}\norm{f(z)}_{p}^{p}d\abs{\mu}(z) 
\\&\leq \norm{T_{\abs{\mu}}}_{\mathcal{L}(A^p_{\alpha})}
\int_{\Bn}\norm{f(z)}_{p}^{p}dv_{\alpha}(z).
\end{align*} 
\end{proof}

\begin{lem}
\label{CM-Cor}
Let $1<p<\infty$ and suppose that $\mu$ is an $A^p_{\alpha}$ 
matrix--valued Carleson measure.  Let $F\subset \Bn$ be a 
compact set, then
$$
\norm{T_{\mu1_F}f}_{A^p_{\alpha}}\lesssim 
\norm{T_\mu}^{\frac{1}{q}}_{\mathcal{L}(A^p_\alpha)}
\norm{1_F f}_{L^p(\Bn,\Cd;\mu)},
$$
where $q=\frac{p}{p-1}$.
\end{lem}
\begin{proof}
It is clear $T_{\mu1_F}f$ is a bounded analytic function for any 
$f\in A^p_{\alpha}$ since $F$ is compact 
and $\mu$ is a finite 
measure.  As in the proof of the previous lemma, there holds
\begin{align*}
\abs{\ip{T_{\mu 1_F} f}{g}_{A^2_{\alpha}}} &=  
\abs{\int_{\Bn} 
\ip{M_{\mu}(w)1_F(w)f(w)}{g(w)}_{\Cd}\,d\tau_{\mu}(w)}
\\&= \ip{1_{F}f}{g}_{L^{2}(\Bn,\Cd;\mu)}
\\&\leq \norm{1_{F}f}_{L^{p}(\Bn,\Cd;\mu)}\norm{g}_{L^{q}(\Bn,\Cd;\mu)}
\\&\lesssim \norm{T_\mu}^{\frac{1}{q}}_{\mathcal{L}(A^p_\alpha)}
\norm{1_{F}f}_{L^{p}(\Bn,\Cd;\mu)}\norm{g}_{A^q_\alpha}.
\end{align*}
Note that in the above we used Proposition \ref{prop:holder}.
\end{proof}

For a Carleson measure $\mu$ and $1<p<\infty$ and for $f\in L^p(\Bn,\Cd;\mu)$ we
also define
$$
P_\mu f(z):=\int_{\Bn}\frac{d\mu(w)f(w)}{(1-\overline{w}z)^{n+1+\alpha}}\,.
$$
It is easy to see based on the computations above that $P_\mu$ is a bounded 
operator from $L^p(\Bn,\Cd;\mu)$ to $A^p_{\alpha}$ and 
$T_\mu=P_{\mu}\circ \imath_p$.

\section{Approximation By Localized Compact Operators}

In this section, we will show that every operator in the Toeplitz algebra can be
approximated by sums of localized compact operators. Along with some other 
estimates, this will help us approximate the essential norm of operators in
the Toeplitz algebra. In particular, the goal of this section will be to prove 
the following Theorem:

\begin{thm}\label{thm:mainSec3}
Let \begin{math}S\in \mathcal{T}_{p,\alpha}\end{math}, 
$\mu$ be a $A_{p}^{\alpha}$ matrix--valued
Carleson measure and 
\begin{math}\epsilon > 0\end{math}. Then there are Borel sets 
\begin{math}F_{j}\subset 
G_{j} \subset \Bn\end{math}
such that 
\begin{description}
   \item[(i)]\begin{math}\Bn = \cup_{j=1}^{\infty} F_{j}\end{math};
   \item[(ii)]\begin{math}F_{j}\cap F_{k} = \emptyset \end{math} if $j\neq k$;
   \item[(iii)] each point of $\Bn$ lies in no more than $N=N(n)$ of the sets 
   $G_{j}$;
   \item[(iv)]\begin{math}\textnormal{diam}_{\beta}G_{j}\leq 
   \mathfrak{d}(p,S,\epsilon)\end{math} for all $j$, and
   \begin{align*}
   \left\|ST_{\mu}-
   \sum_{j=1}^{\infty}M_{1_{F_{j}}}S
   T_{\mu 1_{G_{j}}}\right\|_{\fl(A_{\alpha}^{p},L_{\alpha}^{p})}<\epsilon.
  \end{align*} 
\end{description}
\end{thm}

To prove this, we prove several estimates and put them together at the end of 
this section to prove Theorem \ref{thm:mainSec3}.

\begin{lem}\label{lem:Tech1}
Let $1<p<\infty$, $\alpha>-1$, and $\mu$ be a matrix--valued  
Carleson measure. Suppose that \begin{math}F_{j},K_{j}\subset \Bn\end{math} are 
Borel sets such that $\{F_{j}\}_{j=1}^{\infty}$ are pairwise disjoint 
and 
\begin{math}\beta(F_{j},K_{j})>\sigma \geq 1 \text{ for all }j\end{math}. 
If \begin{math}0 < \gamma < \min\left\{\frac{1}{p(n+1+\alpha)},
\frac{p-1}{p}\right\}\end{math}, 
then
\begin{align*}
\int_{\Bn}\sum_{j=1}^{\infty}1_{F_{j}}(z)1_{K_{j}}(w)
\frac{(1-|w|^{2})^{-1/p}}{|1-\overline{z}w|^{n+1+\alpha}} d|\mu|(w) \lesssim
\|T_{\mu}\|_{\mathcal{L}(A_{\alpha}^{p})}
(1-\delta^{2n})^{\gamma}(1-|z|^{2})^{1/p}.
\end{align*}
\end{lem}

\begin{proof}
This is a consequence of \cite{MSW}*{Lemma 3.3}, and Corollary \ref{cor:car1}.
\end{proof}

\begin{lem}\label{Tech2}
Let $1<p<\infty$ and $\mu$ be a matrix--valued \begin{math}\apa\end{math} 
Carleson measure. Suppose that 
\begin{math}F_{j},K_{j}\subset \Bn\end{math} are Borel sets, 
\begin{math}a_{j}\in L_{M^d}^{\infty}\end{math}, and 
$b_j\in L^{\infty}_{M_d}(\tau_\mu)$. 
\begin{description}
  \item[(i)] \begin{math}\beta(F_{j},K_{j}) \geq\sigma\geq 1\end{math};
  \item[(ii)] \begin{math}\supp a_{j}\subset F_{j}\end{math} and 
	       \begin{math}\supp b_{j}\subset K_{j}\end{math};
  \item[(iii)]  Every \begin{math}z\in\Bn\end{math} belongs to at most 
  \begin{math}N\end{math} of the sets
	       \begin{math}F_{j}\end{math}.
\end{description}
Then \begin{math}\sum_{j=1}^{\infty}M_{a_{j}}P_{\mu}M_{b_{j}}\end{math} is a 
bounded 
operator from \begin{math}A_{\alpha}^{p}\end{math} to 
\begin{math}L_{\alpha}^{p}\end{math} and there is a function \begin{math}
\beta_{p,\alpha}(\sigma)\to 0\end{math} when 
\begin{math}\sigma\to\infty\end{math} 
such that:

\begin{equation}
\label{Tech2-Est1}
\norm{\sum_{j=1}^{\infty} M_{a_j} P_{\mu} M_{b_j}f}_{L^p_{\alpha}}\leq 
N\beta_{p,\alpha}(\sigma)
\norm{T_\mu}_{\mathcal{L}(A^p_\alpha)}\norm{f}_{A^p_{\alpha}},
\end{equation}
and for every \begin{math}f\in\apa\end{math}
\begin{equation}
\label{Tech2-Est2}
\sum_{j =1}^{\infty}\norm{M_{a_j} P_{\mu} M_{b_j} f}_{L^p_{\alpha}}^p\leq 
N\beta_{p,\alpha}^p(\sigma)
\norm{T_\mu}_{\mathcal{L}(A^p_\alpha)}^p\norm{f}^{p}_{A^p_{\alpha}}.
\end{equation}
\end{lem}

\begin{proof}
Since $\mu$ is a matrix--valued Carleson measure for $A^p_{\alpha}$, 
$\kappa_{(p,d)}$ is bounded, with 
\begin{math}\kappa_{(p,d)}\simeq 
\norm{\abs{\mu}}^{\frac{1}{p}}_{\textnormal{RKM}}\simeq
\norm{T_{\mu}}_{\fl(\apa)}\simeq
\norm{T_{\abs{\mu}}}_{\fl(\apa)} 
\end{math}
it is enough to prove the following two estimates:
\begin{equation}
\label{Tech2-Est1-Red}
\norm{\sum_{j=1}^{\infty} M_{a_j} P_{\mu} M_{b_j}f}_{L^p_\alpha}\leq N
\psi_{p,\alpha}(\delta)
\norm{T_{\mu}}_{\mathcal{L}(A^p_\alpha)}^{1-\frac{1}{p}}
\norm{f}_{L^p(\Bn,\Cd;\abs{\mu})},
\end{equation}
and
\begin{equation}
\label{Tech2-Est2-Red}
\sum_{j =1}^{\infty}\norm{M_{a_j} P_{\mu} M_{b_j} f}_{L^p_\alpha}^p\leq 
N\psi_{p,\alpha}^p(\delta)
\norm{T_\mu}_{\mathcal{L}(A^p_\alpha)}^{p-1}
\norm{f}^{p}_{L^p(\Bn,\Cd;\abs{\mu})}
\end{equation}
where $\delta=\tanh\frac{\sigma}{2}$ and $\psi_{p,\alpha}(\delta)\to 0$ 
as $\delta\to 1$.  
Estimates \eqref{Tech2-Est1-Red} and \eqref{Tech2-Est2-Red} imply 
\eqref{Tech2-Est1} and \eqref{Tech2-Est2}
via an application of the matrix--valued Carleson Embedding
Theorem, Corollary \ref{cor:car1}.

First, consider the case when $N=1$, and so the sets $\{F_j\}_{j=1}^{\infty}$\
are pairwise disjoint.  Set
$$
\Phi(z,w)=\sum_{j=1}^{\infty} 1_{F_j}(z) 1_{K_j}(w) 
\frac{1}{\abs{1-\overline{z}w}^{n+1+\alpha}}.
$$
Suppose now that $\norm{a_j}_{L_{M^{d}}^\infty}$ and
$\norm{b_j}_{L_{M^{d}}^\infty(\tau_{\mu})}\leq 1$.  
There holds:
\begin{align*}
\norm{\sum_{j=1}^{\infty} M_{a_j} P_{\mu} M_{b_j} f(z)}_{p} &=  
\norm{\sum_{j=1}^{\infty}a_j(z)\int_{\Bn} 
\frac{b_j(w) f(w)}{(1-\overline{w} z)^{n+1+\alpha}} \,d\mu(w)}_{p}
\\&=\norm{\sum_{j=1}^{\infty}a_j(z)\int_{\Bn} 
\frac{M_{\mu}(w)b_j(w) f(w)}{(1-\overline{w} z)^{n+1+\alpha}} 
d\tau_{\mu}(w)}_{p}
\\& \leq \int_{\Bn} \Phi(z,w) \norm{M_{\mu}(w)}\norm{f(w)}_{p}\,d\tau_{\mu}(w)
\\& \simeq \int_{\Bn} \Phi(z,w) \norm{f(w)}_{p}\,d\abs{\mu}(w).
\end{align*}
We will show that that the operator with kernel $\Phi(z,w)$ is 
bounded from \begin{math}L^{p}(\Bn,\C;\abs{\mu})\end{math} into
\begin{math}L_{\alpha}^{p}\end{math} with norm controlled by a 
constant,
\begin{math}C(n,\alpha,p)\end{math} times 
\begin{math}\psi_{p,\alpha}(\delta)
\norm{T_{{\mu}}}_{\fl(\apa)}^{1-\frac{1}{p}}\end{math}. Assuming 
this is true, there holds:
\begin{align*}
\norm{\sum_{j=1}^{\infty} M_{a_j} P_{\mu} M_{b_j}f}_{L^p_\alpha} &=
\int_{\Bn}\norm{\sum_{j=1}^{\infty} M_{a_j} P_{\mu} M_{b_j} 
f(z)}_{p}^{p}dv_{\alpha}(z)
\\&\leq\int_{\Bn}\lp\int_{\Bn} \Phi(z,w) 
\norm{f(w)}_{p}\,d\abs{\mu}(w)\rp^{p}
dv_{\alpha}(z)
\\&\leq C(n,\alpha,p)\psi_{p,\alpha}(\delta)
\norm{T_{{\mu}}}_{\fl(\apa)}^{1-\frac{1}{p}}
\int_{\Bn}\norm{f(z)}_{p}^{p}d\abs{\mu}(z).
\end{align*}

We use Schur's Test to prove that this operator is bounded. 
Set $h(z)=(1-\abs{z}^2)^{-\frac{1}{pq}}$ and observe that Lemma 
\ref{lem:Tech1} gives
\begin{align*}
\int_{\Bn} \Phi(z,w) h(w)^q \,d\abs{\mu}(w)\lesssim 
\norm{T_\mu}_{\mathcal{L}(A^p_\alpha)}
(1-\delta^{2n})^{\gamma}h(z)^{q}.
\end{align*}
Using Lemma \ref{Growth}, there holds 
\begin{align*}
\int_{\Bn}\Phi(z,w)h(z)^{p}dv_{\alpha}(z) &= 
\int_{\Bn}\sum_{j=1}^{\infty} 1_{F_j}(z) 1_{K_j}(w) 
\frac{(1-\abs{z}^{2})^{-\frac{1}{q}}}{\abs{1-\overline{z}w}^{n+1+\alpha}}
dv_{\alpha}(z)
\\&\leq \int_{\Bn}\sum_{j=1}^{\infty} 1_{F_j}(z) 1_{K_j}(w) 
(1-\abs{w}^{2})^{-\frac{1}{q}-\alpha}dv_{\alpha}(z)
\\&\lesssim h(w)^{p}.
\end{align*}

Therefore, Schur's Lemma says that the operator with kernel $\Phi(z,w)$ is 
bounded
from $L^p(\Bn,\Cd;\abs{\mu})$ to $L^p_{\alpha}(\Bn;\mathbb{C})$ with norm 
controlled by a constant $C(n,\alpha,p)$ times 
\begin{math}\psi_{p,\alpha}(\delta) 
\norm{T_\mu}_{\mathcal{L}(A^p_\alpha)}^{1-\frac{1}{p}}
\end{math}. 

This gives \eqref{Tech2-Est1-Red} when $N=1$.  
Since the sets $F_j$ are disjoint in this case, then we also have 
\eqref{Tech2-Est2-Red} because
$$
\sum_{j =1}^{\infty}\norm{M_{a_j} P_{\mu} M_{b_j} f}_{L^p_{\alpha}}^p
=\norm{\sum_{j=1}^{\infty} 
M_{a_j}P_{\mu} M_{b_j} f}_{L^p_{\alpha}}^p.
$$

Now suppose that $N>1$.  Let $z\in \Bn$ and let 
$S(z)=\left\{j: z\in F_j\right\}$, ordered according to the index $j$.
Each $F_j$ admits a disjoint decomposition $F_j=\bigcup_{k=1}^{N} A_{j}^{k}
$ where $A_{j}^{k}$ is the set of $z\in F_j$
such that
$j$ is the $i^{\textnormal{th}}$ element of $S(z)$.
Then, for $1\leq k\leq N$ the sets $\{A_{j}^k: j\geq 1\}$ are pairwise 
disjoint.
Hence, we can apply the computations obtained above to conclude that
\begin{eqnarray*}
\sum_{j=1}^{\infty} \norm{M_{a_j} P_{\mu} M_{b_j} f}_{L^p_{\alpha}}^p & = &
\sum_{j=1}^{\infty} \sum_{k=1}^{N} \norm{M_{a_j 1_{A^k_j}} P_\mu 
M_{b_j} f}_{L^p_{\alpha}}^p\\
 & = & \sum_{k=1}^{N}\sum_{j=1}^{\infty} \norm{M_{a_j 1_{A^k_j}} 
P_\mu M_{b_j} f}_{L^p_{\alpha}}^p\\
& \lesssim & N 
\psi_{p,\alpha}^p(\delta)\norm{T_\mu}_{\mathcal{L}(A^p_\alpha,)}^{p-1}
\norm{f}_{A^p_\alpha}^p.
\end{eqnarray*}
This gives \eqref{Tech2-Est2-Red}, and \eqref{Tech2-Est1-Red} follows from 
similar computations.
\end{proof}

\begin{lem}
Let \begin{math}1<p<\infty\end{math} and \begin{math}\sigma \geq 1 \end{math}.
Suppose that \begin{math}a_{1},\cdots,a_{k}\in L_{M^d}^{\infty}\end{math}  with
norm at most 1 and that \begin{math}\mu\end{math} is a matrix--valued
Carleson measure. Consider the covering of 
\begin{math}\Bn\end{math} given by Lemma \ref{lem:mswLem2.6} for these values of
\begin{math}k\end{math} and \begin{math}\sigma\geq 1\end{math}. Then there is a 
positive constant 
\begin{math}C(p,k,n,\alpha)\end{math} such that:
\begin{align*}
\left\|\left[\prod_{i=1}^{k}T_{a_{i}}\right]T_{\mu}-
\sum_{j=1}^{\infty}M_{1_{F_{0,j}}}\left[\prod_{i=1}^{k}\right]
T_{\mu 1_{F_{k+1,j}}}\right\|_{\mathcal{L}(\apa)}
\lesssim \beta_{p,\alpha}(\sigma)\|T_{\mu}\|_{\mathcal{L}(\apa)}
\end{align*}
where \begin{math}\beta_{p,\alpha}(\sigma)\to 0\end{math} as 
\begin{math}\sigma\to\infty\end{math}.
\end{lem}

\begin{proof}
First note that the quantity:
\begin{align*}
\left\|\left[\prod_{i=1}^{k}T_{a_{i}}\right]T_{\mu}-
\sum_{j=1}^{\infty}M_{1_{F_{0,j}}}\left[\prod_{i=1}^{k}\right]
T_{\mu 1_{F_{k+1,j}}}\right\|_{\mathcal{L}(\apa)},
\end{align*}
is dominated by the sum of
\begin{equation}
\label{Step1}
\norm{\left[ \prod_{i=1}^{k}T_{a_i} \right] 
T_{\mu}-\sum_{j=1}^\infty M_{1_{F_{0,j}}}
\left[ \prod_{i=1}^{k} T_{a_i1_{F_{i,j}}} \right]
T_{1_{F_{k+1,j}}\mu}}_{\mathcal{L}
(A^p_{\alpha},L^p_{\alpha})}
\end{equation}
and
\begin{align}
\label{Step2}
\norm{\sum_{j=1}^\infty M_{1_{F_{0,j}}}\left[ \prod_{i=1}^{k}T_{a_i}  \right] 
T_{\mu1_{F_{k+1,j}}}-
\sum_{j=1}^\infty 
M_{1_{F_{0,j}}}\left[ \prod_{i=1}^{k} T_{a_i 1_{F_{i,j}}}  \right]
T_{\mu1_{F_{k+1,j}}}}_{\mathcal{L}(A^p_{\alpha}, L^p_{\alpha})}.
\end{align}
Therefore, we only need to prove that the quantities in 
\eqref{Step1} and \eqref{Step2} are controlled by  
$\beta_{p,\alpha}(\sigma)\|T_{\mu}\|_{\mathcal{L}(\apa)}$.

For $0\leq m\leq k+1$, define the operators 
$S_m\in \mathcal{L}\left(A_{\alpha}^p,L_{\alpha}^p\right)$ by
$$
S_m=\sum_{j=1}^{\infty} M_{1_{F_{0,j}}}\left[ \prod_{i=1}^{m} T_{1_{F_{i,j}}a_i}
\prod_{i=m+1}^{k} T_{a_i} \right] T_\mu.
$$
Clearly we have
$S_0=\sum_{j=1}^{\infty} M_{1_{F_{0,j}}} \left[\prod_{i=1}^{k} T_{a_i}\right] T_\mu
= \left[\prod_{i=1}^{k} T_{a_i}\right]T_\mu$, with convergence in the strong operator 
topology.  Similarly, we have
$$
S_{k+1}=\sum_{j=1}^{\infty} M_{1_{F_{0,j}}}\left[\prod_{i=1}^{k} 
T_{a_i1_{F_{i,j}}} \right] T_{\mu1_{F_{k+1,j}}}.
$$
When $0\leq m\leq k-1$, a simple computation gives that
$$
S_m-S_{m+1}=\sum_{j=1}^{\infty} M_{1_{F_{0,j}}} 
\left[\prod_{i=1}^{m}T_{1_{F_{i,j}}a_i} \right]
T_{1_{F_{m+1,j}^c}a_{m+1}}  \left[ \prod_{i=m+2}^{k}T_{a_i}\right]  T_\mu.
$$
Here, of course, we should interpret this product as the identity when the lower index
is greater than the upper index.  Take any $f\in A^p_{\alpha}$ and apply Lemma 
\ref{Tech2}, in particular \eqref{Tech2-Est2}, Lemma \ref{lem:mswLem2.6} and some 
obvious estimates to see that
\begin{eqnarray*}
\norm{\left(S_m-S_{m+1}\right)f}_{L^p_{\alpha}}^p & \leq &
C(p)^{pm}\sum_{j=1}^{\infty}\norm{M_{1_{F_{m,j}}a_m}P_\alpha M_{1_{F_{m+1,j}^c}a_{m+1}}
\left[\prod_{i=m+2}^k T_{a_i} \right] T_\mu f}_{L^p_{\alpha}}^p\\
& \leq & C(p)^{pm}N\beta_{p,\alpha}^p(\sigma)
\norm{\left[\prod_{i=m+2}^k T_{a_i} \right] T_\mu f}_{L^p_{\alpha}}^p\\
& \leq & C(p)^{p(k-1)}N\beta_{p,\alpha}^p(\sigma)
\norm{T_\mu}_{\mathcal{L}(A^p_{\alpha})}^p\norm{f}_{A^p_{\alpha}}^p.
\end{eqnarray*}
Also,
$$
S_{k}-S_{k+1}=\sum_{j=1}^{\infty} M_{1_{F_{0,j}}} 
\left[ \prod_{i=1}^{k}T_{1_{F_{i,j}}a_i} \right] T_{\mu1_{F_{k+1,j}^c}},
$$
and again applying Lemma \ref{Tech2}, and in particular \eqref{Tech2-Est2}, we 
find that
$$
\norm{\left(S_k-S_{k+1}\right)f}_{L^p_{\alpha}}^p \leq C_p^{pk}N
\beta_{p,\alpha}^p(\sigma)
\norm{T_\mu}_{\mathcal{L}(A^p_{\alpha})}^p
\norm{f}^p_{A^p_{\alpha}}.
$$
Since $N=N(n)$, we have the following estimates for $0\leq m\leq k$,
$$
\norm{\left(S_m-S_{m+1}\right)f}_{L^p_{\alpha}}\lesssim \beta_{p,\alpha}
(\sigma)\norm{T_\mu}_{\mathcal{L}(A^p_{\alpha})}\norm{f}_{A^p_{\alpha}}.
$$
But from this it is immediate that \eqref{Step1} holds,
$$
\norm{\left(S_0-S_{k+1}\right)f}_{L^p_{\alpha}}\leq \sum_{m=0}^{k}
\norm{\left(S_m-S_{m+1}\right)f}_{L^p_{\alpha}}\lesssim \beta_{p,\alpha}
(\sigma)
\norm{T_\mu}_{\mathcal{L}(A^p_{\alpha})}\norm{f}_{A^p_{\alpha}}.
$$

The idea behind \eqref{Step2} is similar.  For $0\leq m\leq k$, define the 
operator
$$
\tilde{S}_m=\sum_{j=1}^{\infty} M_{1_{F_{0,j}}}
\left[\prod_{i=1}^{m} T_{1_{F_{i,j}}a_i}\prod_{i=m+1}^{k} T_{a_i}\right]
T_{\mu1_{F_{k+1,j}}},
$$
so we have
\begin{eqnarray*}
\tilde{S}_0 & = & \sum_{j=1}^{\infty} M_{1_{F_{0,j}}}
\left[\prod_{i=1}^{k} T_{a_i}\right] T_{\mu1_{F_{k+1,j}}}\\
\tilde{S}_k & = & \sum_{j=1}^{\infty}M_{1_{F_{0,j}}}
\left[\prod_{i=1}^{k}T_{a_i1_{F_{i,j}}}\right]T_{\mu1_{F_{k+1,j}}}.
\end{eqnarray*}
When $0\leq m\leq k-1$, a simple computation gives
$$
\tilde{S}_m-\tilde{S}_{m+1}=\sum_{j=1}^{\infty}M_{1_{F_{0,j}}}
\Big[\prod_{i=1}^{m} T_{1_{F_{i,j}}a_i}\Big]
T_{1_{F_{m+1,j}^c}a_{m+1}}  \left[\prod_{i=m+2}^{k}T_{a_i}\right]  
T_{\mu1_{F_{k+1,j}}}.
$$
Again, applying obvious estimates and using Lemma \ref{Tech2} one concludes that
\begin{eqnarray*}
\norm{\left(\tilde{S}_m-\tilde{S}_{m+1}\right)f}^p_{L_{\alpha}^p} & 
\leq & C(p)^{p(k-1)}\beta_{p,\alpha}^{p}(\sigma)
\sum_{j=1}^{\infty}\norm{T_{\mu1_{F_{k+1,j}}}f}_{A^p_{\alpha}}^{p}\\
 & \leq & C(p)^{p(k-1)}\beta_{p,\alpha}^{p}(\sigma) \norm{T_{\mu}}_{\mathcal{L}
 (A^p_{\alpha})}^{\frac{p}{q}} \sum_{j=1}^{\infty} 
 \norm{1_{F_{k+1,j}} f}_{L^p(\Bn,\Cd;\mu)}^p\\
 & \leq & C(p)^{p(k-1)}\beta_{p,\alpha}^{p}(\sigma) \norm{T_{\mu}}_{\mathcal{L}
 (A^p_{\alpha})}^{\frac{p}{q}} \norm{f}_{L^p(\Bn,\Cd;\mu)}^p\\
 & \leq & NC(p)^{p(k-1)}\beta_{p,\alpha}^{p}(\sigma)
 \norm{T_{\mu}}_{\mathcal{L}(A^p_{\alpha})}^{\frac{p}{q}+1}  
\norm{f}_{A^ p_{\alpha}}^p.
\end{eqnarray*}
Here the second inequality uses Lemma \ref{CM-Cor}, the next inequality uses 
that the 
sets $\{F_{k+1,j}\}_{j=1}^{\infty}$
form a covering of $\Bn$ with at most $N=N(n)$ overlap, and the last inequality 
uses 
Lemma \ref{lem:car}.
Summing up, for $0\leq m\leq k-1$ we have
$$
\norm{\left(\tilde{S}_m-\tilde{S}_{m+1}\right)f}_{L_{\alpha}^p}\lesssim 
\beta_{p,\alpha}(\sigma)\norm{T_{\mu}}_{\mathcal{L}(A^p_{\alpha})}
\norm{f}_{A^p_{\alpha}},
$$
which implies
$$
\norm{\left(\tilde{S}_0-\tilde{S}_{k}\right)f}_{L_{\alpha}^p}\leq 
\sum_{m=0}^{k-1}\norm{\left(\tilde{S}_m-\tilde{S}_{m+1}\right)f}_{L_{\alpha}^p}
\lesssim \beta_{p,\alpha}(\sigma)
\norm{T_{\mu}}_{\mathcal{L}(A^p_{\alpha})}\norm{f}_{A^p_{\alpha}},
$$
giving \eqref{Step2}.
\end{proof}

\begin{lem}\label{lem:msw3.6}Let 
\begin{align*}
S=\sum_{i=1}^{m}\left[\prod_{l=1}^{k_{i}}T_{a_{l}^{i}}\right]T_{\mu_{i}}
\end{align*}
where \begin{math}a_{j}^{i}\in L_{M^{d}}^{\infty}\end{math}.
Let \begin{math}k=\max_{1\leq i \leq m}\{k_{i}\}
\end{math} and let \begin{math}\mu_{i}\end{math} be matrix-valued measures such 
that
\begin{math}\abs{\mu_{i}}\end{math} are Carleson. 
Given \begin{math}\epsilon >0\end{math}, there is 
\begin{math}\sigma=\sigma(S,\epsilon)\geq 1\end{math} such that
if \begin{math} \{F_{i,j}\}_{j=1}^{\infty}\end{math} and 
\begin{math}0\leq i \leq k+1 \end{math} are the sets 
given by Lemma \ref{lem:mswLem2.6} for these values of 
\begin{math}\sigma\end{math} 
and \begin{math}k \end{math}, then
\begin{align*}
\left\|S-\sum_{j=1}^{\infty}M_{1_{F_{0,j}}}
\sum_{i=1}^{m}\left[\prod_{l=1}^{k_{i}}T_{a_{l}^{i}}\right]
T_{\mu_{i}1_{F_{k+1,j}}}\right\|_{\mathcal{L}(\apa\to L_{\alpha}^{p})} 
< \epsilon.
\end{align*}
\end{lem}

\begin{proof}
Each $\mu_i$ is a matrix--valued measure and by Corollary \ref{cor:eleCar}, the 
total variation of each entry of $\mu_{i}$ is a scalar Carleson measure. 
We will use this fact to prove the present claim. Indeed, we can write each
$\mu_i$ as \begin{math}\sum_{j=1}^{d} \sum_{k=1}^{d}\ip{\mu_{i}e_{j}}
{e_{k}}E_{j,k}\end{math}. We now apply the scalar-valued version of this Lemma, 
which is Lemma 3.5 in \cite{MSW}, to each \begin{math}\ip{\mu_{i}e_{j}}
{e_{k}}E_{j,k}\end{math}. We then use linearity and the triangle inequality to 
conclude the result.
\end{proof}

We are finally ready to prove Theorem \ref{thm:mainSec3}.

\begin{proof}
If $S\in\ft_{p,\alpha}$ then we can find a \begin{math}S_{0}=
\sum_{i=1}^{m}\Pi_{l=1}^{k_{i}}T_{a_{l}^{i}}\end{math} 
such that
\begin{align} \label{est:main31}
\left\|S-S_{0}\right\|_{\fl(\apa)}<\epsilon.
\end{align}
We also know, by Lemma \ref{lem:msw3.6},
we can pick $\sigma=\sigma(S_{0},\epsilon)$ and sets $F_{j}=F_{0,j}$ and 
$G_{j}=F_{k+1,j}$ with 
\begin{align*}
\left\|S_{0}T_{\mu}-\sum_{j=1}^{+\infty}M_{1_{F_{j}}}S_{0}
T_{\mu 1_{G_{j}}}\right\|_{\fl(\apa,\lpa)}<\epsilon.
\end{align*}
We know that (i)-(iv) of Theorem \ref{thm:mainSec3} are satisfied by Lemma 
\ref{lem:mswLem2.6}. Note that by the 
triangle inequality, there holds:
\begin{align}\label{est:main32}
\norm{ST_{\mu}-\sum_{j=1}^{+\infty}M_{1_{F_{j}}}
ST_{\mu 1_{G_{j}}}}_{\fl(\apa,\lpa)} 
&\leq\left\|ST_{\mu}-S_{0}T_{\mu}\right\|_{\fl(\apa,\lpa)} 
\\&+
\left\|S_{0}T_{\mu}-\sum_{j=1}^{+\infty}M_{1_{F_{j}}}S_{0}
T_{\mu 1_{G_{j}}}\right\|_{\fl(\apa,\lpa)} 
\\&+
\left\|\sum_{j=1}^{+\infty}M_{1_{F_{j}}}S_{0}T_{\mu 1_{G_{j}}}-
\sum_{j=1}^{+\infty}M_{1_{F_{j}}}S_{0}
T_{\mu 1_{G_{j}}}\right\|_{\fl(\apa,\lpa)}. 
\end{align}
The first two terms are less than $\epsilon$. To control the third term, let 
$f\in\apa$ and recall that the sequence of sets
$\{F_{j}\}_{j=1}^{\infty}$ is disjoint. Then we have
\begin{align*}
\left\|\sum_{j=1}^{\infty}M_{1_{F_{j}}}(S-S_{0})
T_{\mu 1_{G_{j}}}f\right\|_{\lpa}^{p} &=
\sum_{j=1}^{\infty}\left\|M_{1_{F_{j}}}(S-S_{0})
T_{\mu 1_{G_{j}}}f\right\|_{\lpa}^{p}
\\&\leq \epsilon^{p}\sum_{j=1}^{\infty}\left\|
T_{\mu 1_{G_{j}}}f\right\|_{\apa}^{p}
\\&\leq \epsilon^{p}\sum_{j=1}^{\infty}\left\|
T_{\mu 1_{G_{j}}}f\right\|_{\apa}^{p}
\\&\leq \epsilon^{p}\sum_{j=1}^{\infty}\left\|1_{G_{j}}f
\right\|_{A^{p}_{\alpha}}
\\&\leq N\epsilon^{p}\left\|
T_{\mu}\right\|_{\fl(\apa,\lpa)}\left\|f\right\|_{\apa}^{p}.
\end{align*}
Putting together this estimate and estimates 
\eqref{est:main31} and \eqref{est:main32}, 
Theorem \ref{thm:mainSec3} is proven. 
\end{proof}

\section{A Uniform Algebra and its Maximal ideal Space}
Consider the algebra $\mathcal{A}$ of all scalar-valued bounded uniformly 
continuous functions from the metric space $(\Bn,\rho)$ into 
$(\mathbb{C},|\cdot|)$. Furthermore, let $M_{\mathcal{A}}$ be the maximal ideal
space of $\mathcal{A}$. That is, $M_{\mathcal{A}}$ consistists of the multiplicitaive 
linear functionals on $\mathcal{A}$. In \cite{MSW},
the authors prove that if $\mu$ is a complex-valued measure whose variation is 
Carleson, then there is a sequence of functions $B_{k}(\mu)\in \mathcal{A}$ such
that $T_{B_{k}(\mu)} \to T_{\mu}$ in the $\mathcal{L}(\apa(\Bn;\C))$ norm
(see also \cite{Sua}). We 
will prove a natural generalization to the current case 
of matrix-valued measures. In particular, the following holds:

\begin{thm}\label{thm:msw4.7}
Let $1<p<\infty$, $-1<\alpha$, and $\mu$ be a matrix--valued measure such that 
$|\mu|$ is $\apa$-Carleson. Then there is a sequence of matrix-valued measures 
$B_{k}(\mu)$ such that there holds 
$\Langle B_{k}(\mu)e_i, e_j\Rangle_{A_{\alpha}^{2}}\in \mathcal{A}$ and  
$T_{B_{k}(\mu)}\to T_{\mu}$ in  
$\mathcal{L}(\apa)$ norm.
\end{thm}
\begin{rem}
The condition $\Langle B_{k}(\mu)e_i,e_j\Rangle_{A_{\alpha}^{2}}\in \mathcal{A}$
means that every entry of $B_{k}(\mu)$ is in $\mathcal{A}$.
\end{rem}
\begin{proof}
By Corollary \ref{cor:eleCar}, $|\mu_{(i,j)}|$ is a Carleson measure. By
\cite{MSW}*{Theorem 4.7} there exist functions $B_{k}(\mu_{(i,j)})$ in 
$\mathcal{A}$ such that
\begin{align}
\left\|T_{B_{k}(\mu_{(i,j)})} - 
T_{\mu_{(i,j)}}\right\|_{\mathcal{L}(\apa(\Bn;\C))}
\to 0. 
\label{eqn:maxin}
\end{align}

Let 
\begin{math}B_{k}(\mu)=\sum_{i=1}^{d}\sum_{j=1}^{d}B_{k}
(\mu_{(i,j)})E_{(i,j)}\end{math}. Then there holds:
\begin{align*}
\left\|T_{B_{k}(\mu)}-T_{\mu}\right\|_{\mathcal{L}(A_{\alpha}^{p})} &= 
\left\|\sum_{i=1}^{d}\sum_{j=1}^{d}T_{B_{k}(\mu_{(i,j)})E_{(i,j)}}-
\sum_{i=1}^{d}
    \sum_{j=1}^{d}T_{\mu_{(i,j)}E_{(i,j)}}\right\|_{\mathcal{L}(\apa)}
\\&= \left\|\sum_{i=1}^{d}\sum_{j=1}^{d}T_{B_{k}(\mu_{(i,j)})E_{(i,j)}}-
T_{\mu_{(i,j)}E_{(i,j)}}\right\|_{\mathcal{L}(\apa)}
\\&\leq \sum_{i=1}^{d}\sum_{j=1}^{d}\left\|T_{B_{k}(\mu_{(i,j)})E_{(i,j)}}-
T_{\mu_{(i,j)}E_{(i,j)}}\right\|_{\mathcal{L}(\apa)}
\\&= \sum_{i=1}^{d}\sum_{j=1}^{d}\left\|T_{(B_{k}(\mu_{(i,j)})-
\mu_{(i,j)}))E_{(i,j)}}\right\|_{\mathcal{L}(\apa)}
\\&= \sum_{i=1}^{d}\sum_{j=1}^{d}\left\|T_{(B_{k}(\mu_{(i,j)})-
\mu_{(i,j)}))}\right\|_{\mathcal{L}(\apa(\Bn;\C))}.
\end{align*}
This quantity goes to zero as $k\to\infty$ by (\ref{eqn:maxin}). Note that in 
the above we used the fact that $\left\|T_{\phi 
E_{(i,j)}}\right\|_{\mathcal{L}
(\apa)} = \left\|T_{\phi}\right\|_{\mathcal{L}(\apa(\Bn;\C))}$, which
is easy to see. Indeed, if $f\in\mathcal{L}(\apa)$, then 
$T_{\phi E_{(i,j)}}f=\ip{P\phi f}{e_i}_{\Cd}e_j=T_{\phi}(\ip{f}{e_i}_{\Cd})e_j$.

\end{proof}
Let $\mathcal{A}_{d}$ be the set of $d\times d$ matrices with entries in 
$\mathcal{A}$. Theorem \ref{thm:msw4.7} implies the following Theorem:
\begin{thm}
The Toeplitz Algebra $\mathcal{T}_{p,\alpha}$ equals the closed algebra 
generated by $\{T_{a}:a\in\mathcal{A}_{d}\}$.
\end{thm}

We collect some results about $\mathcal{A}$ and $M_{\mathcal{A}}$. Their proofs 
can be found in, for example, \cite{Sua} and \cite{MSW}.

\begin{lem}
\label{invariance}
Let $z,w,\xi\in\Bn$.  Then there is a positive constant that depends 
only on $n$ such that
$$
\rho(\varphi_z(\xi),\varphi_w(\xi))\lesssim \frac{\rho(z,w)}{1-\abs{\xi}^2}.
$$
\end{lem}

\begin{lem}
\label{Extend}
Let $(E,d)$ be a metric space and $f:\Bn\to E$ be a continuous map.  Then $f$ 
admits a 
continuous extension from $M_{\mathcal{A}}$ into $E$ if and only if $f$ is $
(\rho,d)$ 
uniformly continuous  and $\overline{f(\Bn)}$ is compact.
\end{lem}

\begin{lem}
\label{Maximal3}
Let $\{z_\alpha\}$ be a net in $\Bn$ converging to $x\in M_{\mathcal{A}}$.  Then
\begin{description}
\item[(i)] $a\circ\varphi_x\in\mathcal{A}$ for every $a\in\mathcal{A}$.  In 
particular, 
$\varphi_x:\Bn\to M_{\mathcal{A}}$ is continuous;
\item[(ii)] $a\circ\varphi_{z_{\omega}}\to a\circ\varphi_x$ uniformly on compact
sets of
$\Bn$ for every $a\in\mathcal{A}$.
\end{description}
\end{lem}

\subsection{Maps from $M_{\mathcal{A}}$ into $\mathcal{L}(\apa,\apa)$}
The following discussion is similar to the discussion in \cite{MSW}, and the 
proofs and 
``straightforward computations'' are almost exactly like the scalar-valued 
versions. One
important remark is that when using this strategy to
prove ``quantitative'' facts, we implicitly use the fact that $\Cd$ is a finite 
dimensional vector space and so we may ``pull out'' dimensional constants. As an
example, consider the next lemma, Lemma \ref{lem:uzpa}. First, a definition:
\begin{defn}
Define the operator, $U_{z}^{(p,\alpha)}:A_{\alpha}^{p}\to\apa$, by the 
following formula:
\begin{align}\label{eqn:defU}
U^{(p,\alpha)}_z f(w):= 
f(\varphi_z(w))\frac{(1-\abs{z}^2)^{\frac{n+1+\alpha}{p}}}
{(1-w\overline{z})^{\frac{2(n+1+\alpha)}{p}}}
\end{align}
where the argument of $(1-w\overline{z})$ is used to define the root appearing 
above.
\end{defn}
\begin{lem}\label{lem:uzpa}
There holds:
\begin{align*}
\norm{U^{(p,\alpha)}_z f}_{A^p_{\alpha}}=\norm{f}_{A^p_{\alpha}} 
\quad\forall f\in 
A^p_{\alpha},
\end{align*}
and $\,U^{(p,\alpha)}_zU^{(p,\alpha)}_z=Id_{A^p_\alpha}$.
\end{lem}
\begin{proof}
We will use the change of variables formula in Lemma \ref{eqn:cov}. 
There holds
\begin{align*}
\norm{U_{z}^{(p,\alpha)}f}_{\apa}^{p}&=\int_{\Bn}\norm{f(\varphi_z(w))
\frac{(1-\abs{z}^2)^{\frac{n+1+\alpha}{p}}}{(1-w\overline{z})
^{\frac{2(n+1+\alpha)}{p}}}}_{p}^{p}dv_{\alpha}(w)
\\&=\int_{\Bn}\norm{f(\varphi_z(w))}_{p}^{p}\abs{
\frac{(1-\abs{z}^2)^{n+1+\alpha{}}}
{(1-w\overline{z})^{2(n+1+\alpha){}}}}dv_{\alpha}(w)
\\&=\int_{\Bn}\norm{f(\varphi_{z}(w))}_{p}^{p}
\abs{k_{z}^{(2,\alpha)(w)}}^{2}dv_{\alpha}(z)
\\&=\int_{\Bn}\norm{f(w)}_{p}^{p}dv_{\alpha}(w).
\end{align*}
In the last equality, we used the change of variables formula and the fact that 
$\varphi_{z}$ is an 
involution. 
\end{proof}
There are several ways to justify the change of variables used in the last 
lemma. First, we could use the scalar-valued change of variables formula 
directly by appealing to the fact that 
\begin{math}w\mapsto\norm{f(w)}_{p}\end{math}
is in $L_{\alpha}^{p}(\Bn;\C)$. Secondly, we can use the change of variables 
formula for 
the scalar-valued case indirectly by first passing to the definition of the 
vector-valued integral, and then applying the change of variables on each 
summand in the definition. Either way works, and in what follows, the proofs for
the vector-valued theorems can be proven similarly. 
\\
Note that the operator $U_{z}^{(p,\alpha)}$ can be written in the form:
\begin{align*}
(U_{z}^{(p,\alpha)}f)(w)&=\sum_{k=1}^{d}\
\ip{(U_{z}^{(p,\alpha)}f)(w)}{e_{k}}_{\Cd}e_{k}
\\&=\sum_{k=1}^{d}\frac{(1-\abs{z}^2)^{\frac{n+1+\alpha}{p}}}
    {(1-w\overline{z})^{\frac{2(n+1+\alpha)}{p}}}
    \ip{f\circ\varphi (w)}{e_{k}}_{\Cd}e_{k}
\\&=\sum_{k=1}^{d} (U_{z}^{(p,\alpha)}\ip{f}{e_{k}}_{\Cd})(w)e_{k}.
\end{align*}
In the above \begin{math}\ip{f}{e_{k}}_{\Cd}(w)
=\ip{f(w)}{e_{k}}_{\Cd}\end{math}.

For a real number $r$, set
\begin{align*}
J^r_z(w)=\frac{(1-\abs{z}^2)^{r\frac{n+1+\alpha}{2}}}
{(1-w\overline{z})^{r(n+1+\alpha)}}.
\end{align*}
Let $I_d$ be the $d\times d$ identity matrix. Observe that
\begin{align*}
U^{(p,\alpha)}_z f(w)=
(J_z^{\frac{2}{p}}(w)I_d)f(\varphi_z(w))\quad\textnormal{ and }
\quad  U^{(p,\alpha)}_z= T_{J_z^{\frac{2}{p}-1}I_d}U^{(2,\alpha)}_z=
U^{(2,\alpha)}_z 
T_{J_z^{1-\frac{2}{p}}I_d}.
\end{align*}
So, if $q$ is the conjugate exponent of $p$, we have
$$
\left(U^{(q,\alpha)}_z\right)^{*}=U^{(2,\alpha)}_z 
T_{\overline{J_z}^{\frac{2}{q}-1}I_d}
=T_{\overline{J_z}^{1-\frac{2}{q}}I_d}U^{(2,\alpha)}_z.
$$
Then using that $U^{(2,\alpha)}_zU^{(2,\alpha)}_z=Id_{A^2_\alpha}$ and 
straightforward computations, we obtain
$$
\quad \left(U^{(q,\alpha)}_z\right)^{*}U^{(p,\alpha)}_z=T_{b_zI_d}
\quad\textnormal{ and } \quad 
U^{(p,\alpha)}_z\left(U^{(q,\alpha)}_z\right)^{*}=T_{b_zI_d}^{-1},
$$
where
\begin{equation}
\label{bofz}
b_z(w)=\frac{(1-\overline{w}z)^{(n+1+\alpha)\left(\frac{1}{q}-
\frac{1}{p}\right)}}
{(1-\overline{z}w)^{(n+1+\alpha)\left(\frac{1}{q}-\frac{1}{p}\right)}}.
\end{equation}

For $z\in\Bn$ and $S\in\mathcal{L}(A^p_{\alpha})$ we then define 
the map
$$
S_z :=U^{(p,\alpha)}_z S(U^{(q,\alpha)}_z)^{*},
$$
which induces a map $\Psi_S:\Bn\to \mathcal{L}(A^p_{\alpha}, A^p_{\alpha})$ 
given by 
$$\Psi_S(z)=S_z.$$  
We now show how to extend the map $\Psi_S$ continuously to a map from 
$M_\mathcal{A}$ to
$\mathcal{L}(A^p_{\alpha})$ when endowed with both the weak and 
strong operator topologies.

First, observe that $C(\overline{\Bn})\subset\mathcal{A}$ induces a natural 
projection 
$\pi:M_{\mathcal{A}}\to M_{C(\overline{\Bn})}$.  If $x\in M_{\mathcal{A}}$, let
\begin{equation}
\label{bofx}
b_x(w)=\frac{(1-\overline{w}\pi(x))^{(n+1+\alpha)\left(\frac{1}{q}-
\frac{1}{p}\right)}}
{(1-\overline{\pi(x)}w)^{(n+1+\alpha)\left(\frac{1}{q}-\frac{1}{p}\right)}}.
\end{equation}
So, when $z_\omega$ is a net in $\Bn$ that tends to $x\in M_{\mathcal{A}}$,
then $z_{\omega}=\pi(z_{\omega})\to \pi(x)$ in the Euclidean metric, and so we 
have 
$b_{z_{\omega}}\to b_x$
uniformly on compact sets of $\Bn$ and boundedly.  Furthermore,
$$
(U^{(q,\alpha)}_z)^{*}U^{(p,\alpha)}_z=T_{b_z I_d}\to T_{b_x I_d}
\ \textnormal{ and }\ (U^{(p,\alpha)}_z)^{*}  
U^{(q,\alpha)}_z=T_{\overline{b_z}I_d}\to 
T_{\overline{b_x}I_d},
$$
where convergence is in the strong operator topologies of 
$\mathcal{L}(A^p_{\alpha})$ and
$\mathcal{L}(A^q_{\alpha})$, respectively.  If $a\in\mathcal{A}$ then Lemma
\ref{Maximal3}
implies $a\circ \varphi_{z_{\omega}}\to a\circ\varphi_x$ uniformly on compact 
sets of $\Bn$. The above discussion implies that
$$
T_{(a\circ\varphi_{z_{\omega}} ) 
b_{z_{\omega}}I_d}\to T_{(a\circ\varphi_x)b_x I_d}
$$
in the strong operator topology associated with 
$\mathcal{L}(A^p_{\alpha})$.

Recall that we have $k_{z}^{(p,\alpha)}(w)e=
\frac{(1-\abs{z}^2)^{\frac{n+1+\alpha}{q}}}
{(1-\overline{z}w)^{n+1+\alpha}}e$,
with $\norm{k^{(p,\alpha)}_{z} e}_{\apa}\approx 1$, and so
$$
(1-\abs{\xi}^2)^{\frac{n+1+\alpha}{p}}J_z^{\frac{2}{p}}(\xi)e
=(1-\abs{\varphi_z(\xi)}^2)^{\frac{n+1+\alpha}{p}}
\frac{\abs{1-\overline{z}\xi}^{\frac{2}{p}(n+1+\alpha)}}
{(1-\xi\overline{z})^{\frac{2}{p}(n+1+\alpha)}}e
=(1-\abs{\varphi_z(\xi)}^2)^{\frac{n+1+\alpha}{p}}\lambda_{(p,\alpha)}(\xi,z)e.
$$
Here the constant $\lambda_{(p,\alpha)}$ is unimodular, and
\begin{align*}
\ip{f}{\left(U^{(p,\alpha)}_z\right)^* k_\xi^{(q,\alpha)}e}_{A^2_{\alpha}} 
&=\ip{U^{(p,\alpha)}_zf}{ k_\xi^{(q,\alpha)}e}_{A^2_{\alpha}}
\\&=\ip{J_z^{\frac{2}{p}} 
    (f\circ\varphi_z)}{ k_\xi^{(q,\alpha)}e}_{A^2_{\alpha}}
\\&=\ip{J_z^{\frac{2}{p}} (f\circ\varphi_z) 
    \overline{k_\xi^{(q,\alpha)}}}{e}_{A^2_{\alpha}}
\\&=\ip{J_z^{\frac{2}{p}}(\xi)f(\varphi_z(\xi))(1-\abs{\xi}^2)^{\frac{n+1+\alpha}{p}}}
    {e}_{\Cd}
\\&=\ip{f(\varphi_z(\xi))(1-\abs{\varphi_z(\xi)}^2)^{\frac{n+1+\alpha}{p}}
    \lambda_{(p,\alpha)}(\xi,z)}{e}_{\Cd}
\\&=\ip{f}{\overline{\lambda_{(p,\alpha)}(\xi,z)} 
    k_{\varphi_z(\xi)}^{(q,\alpha)}e}_{A^2_{\alpha}}.
\end{align*}
This computation yields
\begin{equation}
\label{Comp}
\left(U^{(p,\alpha)}_z\right)^* k_\xi^{(q,\alpha)}e
=\lambda_{(p,\alpha)}(\xi,z)k_{\varphi_z(\xi)}^{(q,\alpha)}e.
\end{equation}
We use these computations to study the continuity of the above map as a function
of $z$.
\begin{lem}
\label{UniformCon}
Fix $\,\xi\in\Bn$ and $e\in\Cd$.  Then the map 
$z\mapsto \left(U^{(p,\alpha)}_z\right)^* k_\xi^{(q,\alpha)}e$
is uniformly continuous from 
$(\Bn,\rho)$ into $(A^{q}_{\alpha},\norm{\,\cdot\,}_{A^q_{\alpha}})$.
\end{lem}
\begin{proof}
For \begin{math}z,w\in\Bn\end{math}, there holds
\begin{align*}
\norm{\left(U^{(p,\alpha)}_{z}\right)^* k_\xi^{(q,\alpha)}e-
\left(U^{(p,\alpha)}_{w}\right)^* 
k_\xi^{(q,\alpha)}e}_{A_{\alpha}^{q}(\Bn;\Cd)}
=
\norm{\left(U^{(p,\alpha)}_{z}\right)^* k_\xi^{(q,\alpha)}-
\left(U^{(p,\alpha)}_{w}\right)^* k_\xi^{(q,\alpha)}}_{A_{\alpha}^{q}}.
\end{align*}
And now we may apply the scalar-valued version, \cite{MSW}*{Lemma 4.8}. 
\end{proof}

\begin{prop}
\label{WOTCon}
Let $S\in \mathcal{L}(A^p_{\alpha})$.  Then the map $\Psi_S: \Bn 
\to 
\left(\mathcal{L}(A^p_{\alpha}), WOT\right)$ extends continuously 
to 
$M_{\mathcal{A}}$.
\end{prop}
\begin{proof}
Bounded sets in $\mathcal{L}(A^p_{\alpha})$ are metrizable and 
have 
compact closure in the weak operator topology.  Since $\Psi_S(\Bn)$ is bounded, 
by Lemma
\ref{Extend}, we only need to show $\Psi_S$ is uniformly continuous 
from $(\Bn,\rho)$ 
into $\left(\mathcal{L}(A^p_{\alpha}), WOT\right)$, where $WOT$ is
the weak operator topology.  Namely, we need to demonstrate that for 
$f\in A^p_{\alpha}$ and $g\in A^{q}_{\alpha}$ the function 
$z\mapsto \ip{S_z f}{g}_{A^2_{\alpha}}$ is uniformly continuous from 
$(\Bn,\rho)$ into $(\C, \abs{\,\cdot\,})$.

For $z_1,z_2\in\Bn$ we have
\begin{align*}
S_{z_1}-S_{z_2}
&=U_{z_1}^{(p,\alpha)}S(U_{z_1}^{(q,\alpha)})^*-U_{z_2}^{(p,\alpha)}
    S(U_{z_2}^{(q,\alpha)})^*
\\&=U_{z_1}^{(p,\alpha)}S[(U_{z_1}^{(q,\alpha)})^*-(U_{z_2}^{(q,\alpha)})^*]+
    (U_{z_1}^{(p,\alpha)}-U_{z_2}^{(p,\alpha)})S(U_{z_2}^{(q,\alpha)})^*
\\&=A+B.
\end{align*}
The terms $A$ and $B$ have a certain symmetry, and so it is enough to deal with 
either,
since the argument will work in the other case as well.  Observe that
\begin{eqnarray*}
\abs{\ip{Af}{g}_{A^2_{\alpha}}} & \leq & 
\norm{U_{z_1}^{(p,\alpha)}S}_{\mathcal{L}(A^p_{\alpha})}
\norm{[(U_{z_1}^{(q,\alpha)})^*-(U_{z_2}^{(q,\alpha)})^*]f}_{A^p_{\alpha}}
\norm{g}_{A^q_{\alpha}}\\
\abs{\ip{Bf}{g}_{A^2_{\alpha}}} & \leq & 
\norm{(U_{z_1}^{(q,\alpha)})^{*}S}_{\mathcal{L}
(A^p_{\alpha})}
\norm{[(U_{z_1}^{(p,\alpha)})^*-(U_{z_2}^{(p,\alpha)})^*]g}_{A^q_{\alpha}}
\norm{f}_{A^p_{\alpha}}.
\end{eqnarray*}
Since $S$ is bounded and since 
$\norm{U_{z}^{(p,\alpha)}}_{\mathcal{L}(A^p_{\alpha},)}\leq C(p,\alpha)$ for all
$z$, we just need to show the expression
$$
\norm{[(U_{z_1}^{(p,\alpha)})^*-(U_{z_2}^{(p,\alpha)})^*]g}_{A^q_{\alpha}}
$$
can be made small.  It suffices to do this on a dense set of functions,
and in particular we can take the linear span of 
$\left\{k_{w}^{(p,\alpha)}e:w\in\Bn;e\in\Cd\right\}$.
Then we can apply Lemma \ref{UniformCon} to conclude the result.
\end{proof}

This proposition allows us to define $S_{x}$ for 
\begin{math}x\in M_{\mathcal{A}}\setminus\Bn\end{math}. 
We set $S_{x}:=\Psi_{S}(x)$. If $(z_{\omega})$ is a net that converges 
to \begin{math}x\in M_{\alpha}\end{math}, then 
\begin{math}S_{z_{\omega}}\to S_{x}
\end{math} in WOT. In Proposition~\ref{SOTCon}, we will show that if 
\begin{math}S\in\mathcal{T}_{p,\alpha}\end{math}, then this convergence also 
takes place in SOT.

\begin{lem}
\label{inverse}
If $(z_\omega)$ is a net in $\Bn$ converging to $x\in M_{\mathcal{A}}$, then 
$T_{b_{x} I_d}$ is invertible and 
$T_{b_{z_{\omega}} I_d}^{-1}\to T_{b_x I_d}^{-1}$ in
the strong operator topology.
\end{lem}

\begin{proof}
By Proposition \ref{WOTCon} applied to the operator $S=Id_{A^p_\alpha}$ we have 
that
$U_{z_{\omega}}^{(p,\alpha)}\left(U_{z_{\omega}}^{(q,\alpha)}\right)^*=
T_{b_{z_{\omega}}}^{-1}$ has a weak operator limit in 
$\mathcal{L}(A^p_{\alpha},)$, denote this by $Q$. The Uniform Boundedness 
Principle then says that there is a constant $C$ such that 
$\norm{T_{b_{z_{\omega}}}^{-1}}_{\mathcal{L}(A^p_{\alpha})}\leq C$
for all $\omega$.  Then, given $f\in A^p_{\alpha}$ 
and $g\in A^q_{\alpha}$, since we know that
$$
\norm{\left(T_{\overline{b}_{z_{\omega}}I_d}-
T_{\overline{b}_{x}I_d}\right)g}_{A^q_{\alpha}}
\to 0,
$$
there holds
\begin{align*}
\ip{T_{b_x I_d}Qf}{g}_{A^2_{\alpha}}&=\ip{Qf}
{T_{\overline{b}_x I_d}g}_{L^2_{\alpha}} 
\\&= \lim_{\omega}\ip{T^{-1}_{b_{z_{\omega}} I_d}f}
    {T_{\overline{b}_x I_d}g}_{L^2_{\alpha}}
\\&= \lim_{\omega}\left(\ip{T^{-1}_{b_{z_{\omega}}I_d}f}
    {\left(T_{\overline{b}_x I_d}-
    T_{\overline{b}_{z_{\omega}}i}\right)g}_{L^2_{\alpha}}
    +\ip{T^{-1}_{b_{z_{\omega}}I_d}f}
    {T_{\overline{b}_{z_{\omega}}I_d}g}_{L^2_{\alpha}}\right)
\\&= \ip{f}{g}_{L^2_{\alpha}}+\lim_{\omega}\ip{T^{-1}_{b_{z_{\omega}}I_d}f}
    {\left(T_{\overline{b}_x I_d}-
    T_{\overline{b}_{z_{\omega}}I_d}\right)g}_{L^2_{\alpha}}
\\&=\ip{f}{g}_{L^2_\alpha}.
\end{align*}
This gives $T_{b_x I_d}Q=Id_{A^p_{\alpha}}$.  Since taking adjoints is a 
continuous 
operation in the $WOT$, $T_{\overline{b}_{z_{\omega}}I_d}^{-1}\to Q^*$, and 
interchanging the roles of $p$ and $q$, we have $T_{\overline{b}_x I_d}Q^*=
Id_{A^q_{\alpha}}$, which implies that $QT_{b_x I_d}=Id_{A^p_{\alpha}}$. So, 
$Q=T_{b_x I_d}^{-1}$ and $T_{b_{z_{\omega}}I_d}^{-1}\to T_{b_x I_d}^{-1}$ in the
weak operator topology. Finally,
$$
T_{b_{z_{\omega}}I_d}^{-1}-T_{b_x I_d}^{-1}=T_{b_{z_{\omega}}I_d}^{-1}\left
    (T_{b_x I_d} -T_{b_{z_{\omega}}I_d}\right) T_{b_x I_d}^{-1},
$$
and since 
$\norm{T_{b_{z_{\omega}}I_d}^{-1}}_{\mathcal{L}(A^p_{\alpha})}\leq C$ and 
$T_{b_{z_{\omega}}I_d}-T_{b_x I_d}\to 0$ in the strong
operator topology, we also have $T_{b_{z_{\omega}}I_d}^{-1}\to T_{b_x I_d}^{-1}$
in the strong operator topology as claimed.
\end{proof}

\begin{prop}
\label{SOTCon}
If $S\in\mathcal{T}_{p,\alpha}$ and $(z_{\omega})$ is a net in $\Bn$ that tends 
to 
$x\in M_{\mathcal{A}}$, then $S_{z_{\omega}}\to S_x$ in the strong operator 
topology. 
In particular, $\Psi_S:\Bn\to (\mathcal{L}(A^p_{\alpha}), SOT)$ 
extends 
continuously to $M_{\mathcal{A}}$.
\end{prop}
\begin{proof}
First observe that if $A, B\in \mathcal{L}(A^p_{\alpha})$ then
\begin{align*}
(AB)_z &= U^{(p,\alpha)}_zAB(U^{(q,\alpha)}_z)^{*} 
\\&=U^{(p,\alpha)}_zA(U^{(q,\alpha)}_z)^{*}(U^{(q,\alpha)}_z)^{*} 
    U^{(p,\alpha)}_zU^{(p,\alpha)}_z B(U^{(q,\alpha)}_z)^{*}
\\&=A_z T_{b_z I_d}B_z.
\end{align*}
In general, this applies to longer products of operators.

For $S\in \mathcal{T}_{p,\alpha}$ and $\epsilon>0$, by Theorem \ref{thm:msw4.7}
there is a finite sum of finite products of Toeplitz operators with symbols in 
$\mathcal{A}$ such that
$\norm{R-S}_{\mathcal{L}(A^p_{\alpha})}<\epsilon$, and so
$\norm{R_z-S_z}_{\mathcal{L}(A^p_{\alpha})}<C(p,\alpha)\epsilon$.
Passing to the $WOT$ limit we have $\norm{R_x-S_x}_{\mathcal{L}(A^p_{\alpha})}<
C(p,\alpha)\epsilon$
for all $x\in M_{\mathcal{A}}$.
These observations imply that it suffices to prove the Lemma for $R$ alone, and 
then by linearity, it suffices to consider the special case 
$R=\prod_{j=1}^m T_{a_j E_{i,k}}$, 
where $a_j\in\mathcal{A}$. Recall that $E_{i,k}$ is the $d\times d$ matrix 
with a 1 in the $(i,k)$ position and zeros everywhere else. A simple 
computation shows that
$$
U^{(2,\alpha)}_z T_a U^{(2,\alpha)}_z= T_{a\circ\varphi_z}
$$
and more generally,
\begin{eqnarray*}
(T_a)_z & = & U^{(p,\alpha)}_z\left(U^{(q,\alpha)}_z\right)^*
\left(U^{(q,\alpha)}_z\right)^*T_a 
U^{(p,\alpha)}_zU^{(p,\alpha)}_z\left(U^{(q,\alpha)}_z\right)^*\\
& = & U^{(p,\alpha)}_z\left(U^{(q,\alpha)}_z\right)^*
\left(T_{\overline{J}_z^{1-\frac{2}{q}}}U^{(2,\alpha)}_zT_a
U^{(2,\alpha)}_zT_{J_z^{1-\frac{2}{p}}}\right)U^{(p,\alpha)}_z
\left(U^{(q,\alpha)}_z\right)^*\\
& = & T_{b_{z}}^{-1}T_{(a\circ\varphi_z) b_z}T_{b_{z}}^{-1}.
\end{eqnarray*}
We now combine this computation with the observation at the beginning of the 
proposition to see that
\begin{eqnarray*}
\left(\prod_{j=1}^{m} T_{a_j}\right)_z & = & (T_{a_1})_z T_{b_z}\cdots 
T_{b_z}(T_{a_m})_z\\
 & = & T^{-1}_{b_z} T_{(a_1\circ\varphi_z) b_z}
 T^{-1}_{b_z}T_{(a_2\circ\varphi_z) b_z}T^{-1}_{b_z}\cdots
 T^{-1}_{b_z}T_{(a_m\circ\varphi_z) b_z}T^{-1}_{b_z}.
\end{eqnarray*}
But, since the product of $SOT$ nets is $SOT$ convergent, Lemma \ref{inverse} 
and the 
fact that $T_{(a\circ\varphi_{z_\omega})b_{z_{\omega}}}\to 
T_{(a\circ\varphi_x) b_x}$ in
the $SOT$, give
$$
\left(\prod_{j=1}^{m} T_{a_j}\right)_{z_\alpha}\to T^{-1}_{b_x}
T_{(a_1\circ\varphi_x) b_x}T^{-1}_{b_x}T_{(a_2\circ\varphi_x) b_x}T^{-1}_{b_x}
\cdots T^{-1}_{b_x}
T_{(a_m\circ\varphi_x) b_x}T^{-1}_{b_x}.
$$
But this is exactly the statement $R_{z_{\omega}}\to R_x$ in the $SOT$ for the 
operator $\prod_{j=1}^{m} T_{a_j}$, and proves the claimed continuous extension.
\end{proof}

Before continuing, we prove that the
Berezin transform is one-to-one. The following proof is an adaptation of
the corresponding scalar--valued proof found in \cite[Proposition 6.2]{zhu1}.

\begin{lem}
The Berezin transform is one to one. That is, if $\widetilde{T}=0$, then
$T=0$.
\end{lem}
\begin{proof}

Let $T\in\mathcal{L}(A_{\alpha}^{p})$ and suppose that
$\widetilde{T}=0_d$. (Here, of course, $0_d$ is the $d\times d$ zero matrix.) 
Then there holds:
\begin{align*}
0=\ip{T(k^{(p,\alpha)}_ze_i)}{k^{(q,\alpha)}_ze_k}_{A_{\alpha}^{2}}
=\frac{1}{K^{(\alpha)}(z,z)}\ip{T(K^{(\alpha)}_ze_i)}
{K^{(\alpha)}_ze_k}_{A_{\alpha}^{2}}
\end{align*}
for all $z\in\Bn$ and for all $1\leq i,k \leq d$. In particular, there holds:
\begin{align*}
\frac{1}{K^{(\alpha)}(z,z)}\ip{T(K^{(\alpha)}_ze_i)}
{K^{(\alpha)}_ze_k}_{A_{\alpha}^{2}}\equiv 0.
\end{align*}
Consider the function 
\begin{align*}
F(z,w)=\ip{T(K^{(\alpha)}_we_i)}{K^{(\alpha)}_ze_k}_{A_{\alpha}^{2}}.
\end{align*}
This function is analytic in $z$, conjugate analytic in $w$ and $F(z,z)=0$ for 
all $z\in\Bn$. By a standard results for several complex variables (see for 
instance \cite{kra}*{Exercise 3 pg 365}) this implies 
that $F$ is identically $0$. Using the reproducing property, we conclude that
\begin{align*}
F(z,w)=\ip{T(K^{(\alpha)}_we_i)(z)}{e_k}_{\Cd}\equiv 0,
\end{align*}
and hence 
\begin{align*}
T(K^{(\alpha)}_we_i)(z)\equiv 0,
\end{align*}
for every $w\in\Bn$ and $1\leq i \leq d$. 
Since the products $K_we_i$ span $A_{\alpha}^{p}$, we conclude that 
$T\equiv 0$ and the desired result follows.
\end{proof}

\begin{prop}
\label{BerezinVanish}
Let $S\in\mathcal{L}(A^p_{\alpha})$.  Then $B(S)(z)\to 0$ as 
$\abs{z}\to 1$
if and only if\/ $S_x=0$ for all\/ $x\in M_{\mathcal{A}}\setminus\Bn$.
\end{prop}
\begin{proof}
If $z,\xi\in\Bn$, then we have
\begin{align*}
\ip{B(S_z)(\xi)e_{i}}{e_{j}}_{\Cd}
&=\ip{S\left(U^{(q,\alpha)}\right)^*k_\xi^{(p,\alpha)}e_{i}}
    {\left(U^{(p,\alpha)}\right)^*k_\xi^{(q,\alpha)}e_{j}}_{A^2_{\alpha}}
\\&=\lambda_{(q,\alpha)}(\xi,z)\overline{\lambda_{(p,\alpha)}
    (\xi,z)}\ip{Sk_{\varphi_z(\xi)}^{(p,\alpha)}e_{i}}
    {k_{\varphi_z(\xi)}^{(q,\alpha)}e_{j}}_{A^2_{\alpha}}.
\end{align*}
Thus, $\abs{\ip{B(S_z)(\xi)e_{i}}{e_{j}}_{\Cd}}
=\abs{\ip{B(S)(\varphi_z(\xi))e_{i}}{e_{j}}_{\Cd}}$ 
since $\lambda_{(p,\alpha)}$ and 
$\lambda_{(q,\alpha)}$
are unimodular numbers.
For $x\in M_{\mathcal{A}}\setminus\Bn$ and $\xi\in\Bn$ fixed, if $(z_\omega)$ is
a net in $\Bn$ tending to $x$, the continuity of $\Psi_S$ in the $WOT$ and 
Proposition \ref{WOTCon} give that $B(S_{z_\omega})(\xi)\to B(S_x)(\xi)$, and 
consequently 
$\abs{\ip{B(S)(\varphi_{z_\omega}(\xi))e_{i}}{e_{j}}_{\Cd}}\to 
\abs{\ip{B(S_x)(\xi)e_{i}}{e_{j}}_{\Cd}}$.

Now, suppose that $B(S)(z)$ vanishes as $\abs{z}\to 1$.  
Since $x\in M_{\mathcal{A}}\setminus\Bn$ 
and $z_\omega\to x$, we have that $\abs{z_\omega}\to 1$, and similarly 
$\abs{\varphi_{z_{\omega}}(\xi)}\to 1$. Since $B(S)(z)$ vanishes as we approach the 
boundary, $B(S_x)(\xi)=0$, and since $\xi\in\Bn$ was arbitrary and the Berezin transform
is one-to-one, we see that $S_x=0$.

Conversely, suppose that the Berezin transform does not vanish as we approach the 
boundary. Then there is a sequence $\{z_k\}_{k=1}^{\infty}$ in $\Bn$ such that 
$\abs{z_k}\to 1$ and 
$\abs{\ip{B(S)(z_k)e_{i}}{e_{j}}_{\Cd}}\geq\delta>0$, for all 
\begin{math}i,j=1,\ldots ,d\end{math}. Since $M_{\mathcal{A}}$ is compact, we can 
extract a subnet $(z_{\omega})$ of $\{z_k\}_{k=1}^{\infty}$ 
converging in $M_{\mathcal{A}}$ to
$x\in M_{\mathcal{A}}\setminus\Bn$.  The computations above imply 
$\abs{\ip{B(S_x)(0)e_{i}}{e_{j}}_{\Cd}}\geq\delta>0$, which gives that $S_x\neq 0$.
\end{proof}

\section{Characterization of the Essential Norm on \text{$A^p_{\alpha}\ $}
{Weighted Bergman Spaces}}
\label{Characterization}

We have now collected enough tools to provide a characterization of the 
essential norm 
of an operator on $A^p_{\alpha}$. Even more than in the previous sections, this 
section 
uses the arguments of \cite{MSW} in a nearly verbatim way. Fix $\varrho>0$ and 
let $\{w_m\}_{m=1}^{\infty}$ 
and $\{D_m\}_{m=1}^{\infty}$ be the points and sets of Lemma \ref{StandardGeo}.  Define the measure
$$
\mu_{\varrho}:=\sum_{m=1}^\infty v_{\alpha}(D_m) \delta_{w_m}I_d,
$$
The measure 
\begin{align*}
\nu_{\varrho}:=\sum_{m=1}^\infty v_{\alpha}(D_m) \delta_{w_m}
\end{align*}
is well-known to be an $A^p_{\alpha}$ Carleson measure, so the measure 
$\mu_{\vr}$ is also Carleson. This implies that $T_{\mu_\varrho}:A^p_{\alpha}\to
A^p_{\alpha}$ is bounded. The following lemma is easily deduced from
\cite{MSW}*{Lemma 5.1} (in which the authors refer the reader to 
\cites{L,CR} for a proof), and we omit the proof. 

\begin{lem}
$T_{\mu_\varrho}\to Id_{A^p_\alpha}$ on 
$\mathcal{L}(A^p_{\alpha})$ 
when $\varrho\to 0$.
\end{lem}

Now choose 
$0<\varrho\leq 1$ so that $\norm{T_{\mu_\varrho}-Id_{A^p_\alpha}}_{\mathcal{L}
(A^p_{\alpha})}<\frac{1}{4}$ 
and consequently 
$\norm{T_{\mu_\varrho}}_{\mathcal{L}(A^p_{\alpha})}$ and 
$\norm{T_{\mu_\varrho}^{-1}}_{\mathcal{L}(A^p_{\alpha})}$ are less
than 
$\frac{3}{2}$.  Fix this value of $\varrho$, and denote $\mu_\varrho:=\mu$ for 
the rest of the paper.

For $S\in\mathcal{L}(A^p_{\alpha}, A^p_{\alpha})$ and $r>0$, let
$$
\mathfrak{a}_S(r):=\varlimsup_{\abs{z}\to 1}\sup\left\{\norm{Sf}_{A^p_{\alpha}}:
f\in T_{\mu 1_{D(z,r)}}(A^p_{\alpha}), \norm{f}_{A^p_{\alpha}}\leq 1\right\},
$$
and then define
$$
\mathfrak{a}_S:=\lim_{r\to 1}\mathfrak{a}_S(r).
$$
Since for $r_1<r_2$ we have $T_{\mu 1_{D(z,r_1)}}(A^p_{\alpha})\subset 
T_{\mu 1_{D(z,r_2)}}(A^p_{\alpha})$ and
$\mathfrak{a}_S(r)\leq\norm{S}_{\mathcal{L}(A^p_{\alpha})}$, this 
limit is 
well defined. We define two other measures of the size of an operator which are 
given in
a very intrinsic and geometric way:
\begin{eqnarray*}
\mathfrak{b}_{S} & := & \sup_{r>0} \varlimsup_{\abs{z}\to 1}
\norm{M_{1_{D(z,r)}}S} _{\mathcal{L}(A^p_{\alpha},L^p_{\alpha})},\\
\mathfrak{c}_{S} & := & \lim_{r\to 1}
\norm{M_{1_{(r\Bn)^c}}S}_{\mathcal{L}(A^p_{\alpha},L^p_{\alpha})}.
\end{eqnarray*}
In the last definition, for notational simplicity, we let 
$\left(r\Bn\right)^c=\Bn\setminus r\Bn$.  Finally, for $S\in 
\mathcal{L}(A^p_{\alpha})$ recall that
$$
\norm{S}_e=\inf\left\{\norm{S-Q}_{\mathcal{L}(A^p_{\alpha},A^p_{\alpha})}:
Q\textnormal{ is compact}\right\}.
$$

We first show how to compute the essential norm of an operator $S$ in terms of the 
operators $S_x$, where $x\in M_{\mathcal{A}}\setminus\Bn$.

\begin{thm}
\label{EssentialviaSx}
Let $S\in\mathcal{T}_{p,\alpha}$.  Then there exists a constant $C(p,\alpha,n)$ 
such that
\begin{equation}
\label{EssentialNorm&Sx}
\sup_{x\in M_{\mathcal{A}}\setminus\Bn}\norm{S_x}_{\mathcal{L}(A^p_\alpha 
)}\lesssim\norm{S}_e\lesssim 
\sup_{x\in M_{\mathcal{A}}\setminus\Bn}\norm{S_x}_{\mathcal{L} 
(A^p_\alpha)}.
\end{equation}
\end{thm}

\begin{proof}
For $S$ compact, \eqref{EssentialNorm&Sx} is easy to demonstrate.
Since $k_\xi^{(p,\alpha)}\to 0$ weakly as $\abs{\xi}\to 1$, then 
\begin{math}\ip{k_{\xi}^{p,\alpha}e}{f}_{A_{\alpha}^{2}}=
\sum_{i=1}^{d}\ip{e}{e_i}_{\Cd}\ip{k_{\xi}^{p,\alpha}e_{i}}{f}_{A_{\alpha}^{2}}
\to 0
\end{math}
for every $f\in A_{\alpha}^{q}$ and so $k^{p,\alpha}_{\xi}e\to 0$ weakly as 
$\abs{\xi}\to
1$. 
Therefore, 
$\norm{Sk_\xi^{(p,\alpha)}e}_{A^p_{\alpha}}$ 
goes to $0$ as well. Thus, we have
\begin{align}
\label{Vanishing}
\abs{\ip{\widetilde{S}(\xi)e}{h}_{\Cd}}=
\abs{\ip{S k_\xi^{(p,\alpha)}e}{k_\xi^{(q,\alpha)}h}_{A^2_\alpha}}\leq
\norm{Sk_\xi^{(p,\alpha)}e}_{A^p_{\alpha}}
\norm{k_\xi^{(q,\alpha)}h}_{A^q_{\alpha}}\approx
\norm{Sk_\xi^{(p,\alpha)}e}_{A^p_{\alpha}}.
\end{align}
Hence, the compactness of $S$ implies that the Berezin transform vanishes as 
$\abs{\xi}\to 1$.
Then  Proposition \ref{BerezinVanish} gives that $S_x=0$ for all 
$x\in M_{\mathcal{A}}\setminus\Bn$.

Now let $S$ be any bounded operator on $A^p_{\alpha}$ and suppose that $Q$ is a 
compact operator on $A^p_{\alpha}$.
Select $x\in M_{\mathcal{A}}\setminus \Bn$ and a net $(z_\omega)$ in $\Bn$ tending to 
$x$.
Since the maps $U_{z_\omega}^{(p,\alpha)}$ and $U_{z_\omega}^{(q,\alpha)}$ are 
isometries on 
$A^p_{\alpha}$ and $A^q_{\alpha}$,
we have
$$
\norm{S_{z_{\omega}}+Q_{z_{\omega}}}_{\mathcal{L}(A^p_\alpha)}\leq
\norm{S+Q}_{\mathcal{L}(A^p_\alpha)}.
$$
Since $S_{z_{\omega}}+Q_{z_{\omega}}\to S_x$ in the $WOT$, passing to the limit 
we get
$$
\norm{S_x}_{\mathcal{L}(A^p_\alpha)}\lesssim 
\varliminf\norm{S_{z_{\omega}}+
Q_{z_{\omega}}}_{\mathcal{L}(A^p_\alpha)}\leq \norm{S+Q}_{\mathcal{L}
(A^p_\alpha)},
$$
which gives
$$
\sup_{x\in M_{\mathcal{A}}\setminus\Bn}
\norm{S_x}_{\mathcal{L}(A^p_\alpha,)}\lesssim\norm{S}_e,
$$
the first inequality in \eqref{EssentialNorm&Sx}.  It only remains to address 
the last
inequality.  To accomplish this, we will instead 
prove that
\begin{equation}
\label{AlphaControlled}
\mathfrak{a}_S\lesssim \sup_{x\in M_{\mathcal{A}}\setminus\Bn}
\norm{S_x}_{\mathcal{L}(A^p_\alpha)}.
\end{equation}
Then we compare this with the first inequality in \eqref{Redux2}, 
$\norm{S}_e\lesssim\mathfrak{a}_S$, shown below, to obtain
$$
\norm{S}_e\lesssim 
\sup_{x\in M_{\mathcal{A}}\setminus\Bn}\norm{S_x}_{\mathcal{L}(A^p_\alpha)}.
$$
Also note that if \eqref{AlphaControlled} holds, then
\begin{equation}
\label{LastStep}
\mathfrak{a}_S\lesssim\norm{S}_e
\end{equation}
is also true.  We now turn to addressing \eqref{AlphaControlled}.  It suffices 
to demonstrate that
$$
\mathfrak{a}_S(r)\lesssim\sup_{x\in M_{\mathcal{A}}\setminus\Bn}
\norm{S_x}_{\mathcal{L}(A^p_\alpha)}\quad\ \forall r>0.
$$
Fix a radius $r>0$. By the definition of $\mathfrak{a}_S(r)$ there is a sequence
$\{z_j\}_{j=1}^{\infty}\subset\Bn$ tending to
$\partial\Bn$ and a normalized sequence of functions $f_j
\in T_{\mu 1_{D(z_j,r)}} (A^p_{\alpha})$ with
$\norm{Sf_j}_{A^p_{\alpha}}\to\mathfrak{a}_S(r)$.  To each $f_j$ we have a 
corresponding $h_j\in A^p_{\alpha}$, and then
\begin{align*}
f_j(w)&=T_{\mu 1_{D(z_j,r)}}h_j(w) 
\\&= \sum_{w_m\in D(z_j,r)}\frac{v_\alpha(D_m)}
 {(1-\overline{w_m}w)^{n+1+\alpha}}h_j(w_m)
\\&=\sum_{w_m\in D(z_j,r)} 
\frac{(1-\abs{w_m}^2)^{\frac{n+1+\alpha}{q}}}
    {(1-\overline{w_m}w)^{n+1+\alpha}}a_{j,m}
\\& = \sum_{w_m\in D(z_j,r)} k_{w_m}^{(p,\alpha)}(w)a_{j,m},
\end{align*}
where $a_{j,m}=v_\alpha(D_m)(1-\abs{w_m}^2)^{-\frac{n+1+\alpha}{q}}h_j(w_m)$.  
We then have that
$$
\left(U^{(q,\alpha)}_{z_j}\right)^{*}f_j(w)= \sum_{\varphi_{z_j}(w_m)\in D(0,r)} 
k_{\varphi_{z_j}(w_m)}^{(p,\alpha)}(w)a_{j,m}',
$$
where $a_{j,m}'$ is simply the original constant $a_{j,m}$ multiplied by the 
unimodular constant $\lambda_{(q,\alpha)}$.

Observe that the points $\abs{\varphi_{z_j}(w_m)}\leq\tanh r$.  For $j$ fixed, 
arrange the points $\varphi_{z_j}(w_m)$ such that $\abs{\varphi_{z_j}(w_m)}\leq
\abs{\varphi_{z_j}(w_{m+1})}$ and 
$\textnormal{arg}\,\varphi_{z_j}(w_m)\leq\textnormal{arg}\, \varphi_{z_j}(w_{m+1})$.  
Since the M\"obius map $\varphi_{z_j}$ preserves the hyperbolic distance between the 
points $\{w_m\}_{m=1}^{\infty}$ we have for $m\neq k$ that
$$
\beta(\varphi_{z_j}(w_m),\varphi_{z_j}(w_k))=\beta(w_m,w_k)\geq\frac{\varrho}{4}>0.
$$
Thus, there can only be at most $N_j\leq M(\varrho, r)$ points in the collection 
$\varphi_{z_j}(w_m)$ belonging to the disc $D(0,z_j)$.  By passing to a subsequence, we 
can assume that $N_j=M$ and is independent of $j$.

For the fixed $j$, and $1\leq m\leq M$, select scalar--valued
$g_{j,k}\in H^\infty$ with
$\norm{g_{j,k}}_{H^\infty}\leq C(\tanh r,\frac{\varrho}{4})$, such that 
$g_{j,k}(\varphi_{z_j}(w_m))=\delta_{k,m}$,
the Kronecker delta, when $1\leq k\leq M$.
The existence of the functions is easy to 
deduce from a result of Berndtsson \cite{B}, 
see also \cite{Sua}.  We then have
\begin{align*}
\ip{\left(U_{z_j}^{(q,\alpha)}\right)^{*}f_j}{g_{j,k}e}_{A^2_{\alpha}}  
&=  \int_{\Bn}\ip{\left(U_{z_j}^{(q,\alpha)}\right)^{*}f_j(w)}
{g_{j,k}(w)e}_{\Cd}dv_{\alpha}(w)
\\&=\int_{\Bn} \sum_{\varphi_{z_j}(w_m)\in D(0,r)} 
k_{\varphi_{z_j}(w_m)}^{(p,\alpha)}(w)
\ip{a_{j,m}'}{g_{j,k}(w)e}_{\Cd}dv_{\alpha}(w)
\\&= \sum_{\varphi_{z_j}(w_m)\in D(0,r)} 
\left(1-\abs{\varphi_{z_j}(w_m)}^2\right)^{\frac{n+1+\alpha}{q}}\ip{a_{j,m}'} 
{g_{j,k}(\varphi_{z_j}(w_m))e}_{\Cd}\\
 \\&=  
\left(1-\abs{\varphi_{z_j}(w_k)}^2\right)^{\frac{n+1+\alpha}{q}}
\ip{a_{j,k}'}{e}_{\Cd}.
\end{align*}
This expression implies that the sequence $\abs{\ip{a_{j,k}'}{e}_{\Cd}}
\leq C=C(n,p,\varrho,r,\alpha)$ independently of $j$ and $k$,
because $g_{j,k}\in H^\infty$ has norm controlled by 
$C(r,\varrho)$, $\left( U^{(q,\alpha)}_z\right)^*$ is a bounded operator, 
and $f_j$ is a normalized sequence of functions in $A^p_\alpha$.

Now $(\varphi_{z_j}(w_1),\ldots,\varphi_{z_j}(w_M), 
\ip{a_{j,1}'}{e}_{\Cd},\ldots,\ip{a_{j,M}'}{e}_{\Cd})\in\C^{M(n+1)}$ is a 
bounded sequence in $j$,
and passing to a subsequence if necessary, we can assume that converges to a point 
$(v_1,\ldots, v_M,\ip{a_{1}'}{e}_{\Cd},\ldots,\ip{a_{M}'}{e}_{\Cd})$.
Here $\abs{v_k}\leq\tanh r$ and $\abs{a_k'}\leq C$.  This implies
$$
\ip{\left(U^{(q,\alpha)}_{z_j}\right)^{*}f_j}{e}_{\Cd}\to \sum_{k=1}^M 
\ip{k_{v_k}^{(p,\alpha)}a_k'}{e}_{\Cd}
$$
which means that:
\begin{align*}
\lp U_{z_{j}}^{(q,\alpha)}\rp^{\ast}f_{j}\to
\sum_{k=1}^M k_{v_k}^{(p,\alpha)}a_k'=:h
\end{align*}
in the $L^p_{\alpha}$ norm and moreover,
\begin{align*}
\norm{\sum_{k=1}^M k_{v_k}^{(p,\alpha)}a_k'}_{\lpa}=
\lim_{j\to\infty}\norm{\lp U_{z_{j}}^{(q,\alpha)}\rp^{\ast}f_{j}}_{\lpa}.
\end{align*}
Since the operator $U_{z_j}^{(p,\alpha)}$ is isometric and 
$\norm{S_{z_j}}_{\mathcal{L}(A^p_\alpha, A^p_{\alpha})}$
is bounded independently of $j$,
$$
\mathfrak{a}_S(r)=\lim_{j\to\infty}\norm{S f_j}_{A^p_\alpha}=
\lim_{j\to\infty}
\norm{S_{z_j}(U_{z_j}^{(q,\alpha)})^* f_j}_{A^p_\alpha}=
\lim_{j\to\infty}\norm{S_{z_j} h}_{A^p_\alpha}.
$$
Since $\abs{z_j}\to 1$, by using the compactness of $M_{\mathcal{A}}$ it is 
possible to extract a subnet $(z_\omega)$
which converges to some point $x\in M_{\mathcal{A}}\setminus\Bn$.  
Then $S_{z_\omega}h\to S_x h$ in $A^p_{\alpha}$,  so
$$
\mathfrak{a}_S(r)=\lim_{\omega}\norm{S_{z_\omega} h}_{A^p_\alpha}=
\norm{S_xh}_{A^p_\alpha}\lesssim\norm{S_x}_{\mathcal{L}(A^p_\alpha)}\lesssim
\sup_{x\in M_{\mathcal{A}}\setminus\Bn}
\norm{S_x}_{\mathcal{L}(A^p_\alpha)}.
$$
The above limit uses the continuity in the $SOT$ as guaranteed by Proposition 
\ref{SOTCon}.
\end{proof}

\begin{thm}
Let\/ $1<p<\infty$, $\,\alpha>-1$ and\/ $S\in \mathcal{T}_{p,\alpha}$.  Then 
there exist constants depending only on $n$, $p$, and $\alpha$ such that:
$$
\mathfrak{a}_S\approx\mathfrak{b}_S\approx\mathfrak{c}_S\approx\norm{S}_e.
$$
\end{thm}

\begin{proof}
By Theorem \ref{thm:mainSec3} there are Borel sets $F_j\subset G_j\subset\Bn$ 
such that
\begin{description}
\item[(i)] $\Bn=\cup_{j=1}^{\infty} F_j$;
\item[(ii)] $F_j\cap F_k=\emptyset$ if $j\neq k$;
\item[(iii)] each point of $\Bn$ lies in no more than $N(n)$ of the sets $G_j$;
\item[(iv)] $\textnormal{diam}_{\beta}\, G_j\leq 
\mathfrak{d}(p,S,\epsilon)$
\end{description}
and
\begin{equation}
\label{RecallEst}
\norm{ST_\mu-\sum_{j=1}^{\infty} M_{1_{F_j}}
ST_{\mu 1_{G_j}}}_{\mathcal{L}(A^p_{\alpha}, L^p_{\alpha})}<\epsilon.
\end{equation}
Set
$$
S_m=\sum_{j=m}^{\infty} M_{1_{F_j}} S T_{\mu 1_{G_j}}.
$$
Next, we consider one more measure of the size of $S$,
$$
\varlimsup_{m\to\infty} \norm{\sum_{j=m}^{\infty} M_{1_{F_j}} S 
T_{\mu 1_{G_j}}}_{\mathcal{L}(A^p_{\alpha}, L^p_{\alpha})}=
\varlimsup_{m\to\infty}\norm{S_m}_{\mathcal{L}(A^p_{\alpha}, L^p_{\alpha})}.
$$
First some observations.  Since every $z\in\Bn$ belongs to only $N(n)$ sets 
$G_j$, 
Lemma~\ref{CM-Cor} gives
\begin{align}\label{eqn:est1}
\sum_{j=m}^{\infty}\norm{T_{\mu 1_{G_j}}f}_{A^p_{\alpha}}^p\lesssim 
\sum_{j=1}^{\infty} 
\norm{1_{G_j}f}_{L^p(\Bn,\Cd;\mu)}^p \lesssim\norm{f}_{A^p_{\alpha}}^p.
\end{align}
Also, since $T_\mu$ is bounded and invertible, we have that 
$\norm{S}_e\approx\norm{ST_\mu}_e$.  
Finally, we will need to compute both norms in $\mathcal{L}(A^p_{\alpha}, 
A^p_{\alpha})$
and $\mathcal{L}(A^p_{\alpha}, L^p_{\alpha})$.  When necessary, we will denote 
the respective essential norms as $\norm{\,\cdot\,}_{e}$ and 
$\norm{\,\cdot\,}_{ex}$.  It is easy to show that
$$
\norm{R}_{ex}\leq\norm{R}_e\leq\norm{P_\alpha}_{\mathcal{}(L^p_{\alpha},
A^p_{\alpha})}\norm{R}_{ex}.
$$
The strategy behind the proof of the theorem is to demonstrate the following 
string of inequalities
\begin{eqnarray}
\mathfrak{b}_S & \leq & \mathfrak{c}_S\,\,\lesssim \,\,
\varlimsup_{m\to\infty}\norm{S_m}_{\mathcal{L}
(A^p_{\alpha}, L^p_{\alpha})}\,\,\lesssim\,\,\mathfrak{b}_S\label{Redux1}\\
\norm{S}_{e} & \lesssim & \varlimsup_{m\to\infty}
\norm{S_m}_{\mathcal{L}(A^p_{\alpha}, 
L^p_{\alpha})}\,\,\lesssim\,\, \mathfrak{a}_S\,\,\lesssim\,\, 
\norm{S}_e\label{Redux2}.
\end{eqnarray}

The implied constants in all these estimates depend only on $p,\alpha$ and the 
dimension of the domain, $n$ and the dimension of the range, $d$. Combining 
\eqref{Redux1} and \eqref{Redux2} we have the theorem. We prove now the first 
two inequalities in \eqref{Redux2}.

Fix  $f\in A^p_{\alpha}$ of norm 1 and note that
\begin{align}
\norm{S_m f}^p_{L^p_{\alpha}}
&=\sum_{j=m}^{\infty} \norm{M_{1_{F_j}} ST_{\mu 1_{G_j}}f}_{L^p_{\alpha}}^p
\notag
\\&= \sum_{j=m}^{\infty}\left(\frac{\norm{M_{1_{F_j}} 
    ST_{\mu 1_{G_j}}f}_{L^p_{\alpha}}}
    {\norm{T_{\mu 1_{G_j}}f}_{A^p_{\alpha}}}\right)^{p}
    \norm{T_{\mu 1_{G_j}}f}_{A^p_{\alpha}}^p\notag
\\&\leq\sup_{j\geq m}\sup\left\{\norm{M_{1_{F_j}} Sg}^p_{L^p_{\alpha}}: 
    g\in T_{\mu 1_{G_j}}
    (A^p_{\alpha}),\norm{g}_{A^p_{\alpha}}= 1\right\}\sum_{j\geq m}
    \norm{T_{\mu 1_{G_j}}f}_{A^p_{\alpha}}^p\notag
\\&\lesssim \sup_{j\geq m}\sup\left\{\norm{M_{1_{F_j}} Sg}^p_{L^p_{\alpha}}: 
    g\in T_{\mu 1_{G_j}}(A^p_{\alpha}),\norm{g}_{A^p_{\alpha}}= 1\right\}
    \label{Last}.
\end{align}

In the last step we use the estimate in (\ref{eqn:est1}).
Since $\textnormal{diam}_{\beta}\, G_j\leq \mathfrak{d}$, by selecting $z_j\in G_j$ we have
$G_j\subset D(z_j,d)$, 
and so $T_{\mu 1_{G_j}}(A^p_\alpha)\subset T_{1_{\mu D(z_j,\mathfrak{d})}}(A^p_\alpha)$.
Since $z_j$ approaches the boundary, we can select an additional sequence $0
<\gamma_m<1$ tending to $1$ such that $\abs{z_j}\geq\gamma_m$ when $j\geq m$.  
Using \eqref{Last} we find that
\begin{align}
\norm{S_m}_{\mathcal{L}(A_\alpha^p,L_\alpha^p)} 
&\lesssim  \sup_{j\geq m} \sup\left\{\norm{M_{1_{F_j}} Sg}_{L^p_{\alpha}}: 
    g\in T_{\mu 1_{G_j}}(A^p_{\alpha}),\norm{g}_{A^p_{\alpha}}
    = 1\right\}
\\&\lesssim \sup_{\abs{z_j}\geq \gamma_m}\sup
    \left\{\norm{M_{1_{D(z_j,d)}} Sg}_{L^p_{\alpha}}: g\in 
    T_{\mu 1_{D(z_j,\mathfrak{d})}}(A^p_{\alpha}),
    \norm{g}_{A^p_{\alpha}}= 1\right\}
    \label{important}
\\&\lesssim \sup_{\abs{z_j}\geq \gamma_m}\sup\left\{\norm{Sg}_{L^p_{\alpha}}: 
    g\in T_{\mu 1_{D(z_j,\mathfrak{d})}}(A^p_{\alpha}),
    \norm{g}_{A^p_{\alpha}}= 1\right\}
    \notag.
\end{align}
Since $\gamma_m\to 1$ as $m\to \infty$, we get
$$
\varlimsup_{m\to\infty} \norm{S_m}_{\mathcal{L}(A_\alpha^p,L_\alpha^p)}\lesssim 
\mathfrak{a}_S(\mathfrak{d}).
$$
From \eqref{RecallEst} we see that
$$
\norm{ST_\mu}_{ex}\leq\varlimsup_{m\to\infty} 
\norm{S_m}_{\mathcal{L}(A_\alpha^p,L_\alpha^p)}+\epsilon
\lesssim \mathfrak{a}_S(\mathfrak{d})
+\epsilon\lesssim\mathfrak{a}_S+\epsilon,
$$
giving $\norm{ST_\mu}_{ex}\leq\varlimsup_{m\to\infty} 
\norm{S_m}_{\mathcal{L}(A_\alpha^p,L_\alpha^p)}
\lesssim\mathfrak{a}_S$,
since $\epsilon$ is arbitrary.  Therefore,
\begin{equation}
\norm{S}_e\approx\norm{ST_\mu}_{e}\lesssim \norm{ST_\mu}_{ex}\leq \varlimsup_{m}
\norm{S_m}_{\mathcal{L}(A_\alpha^p,L_\alpha^p)}\lesssim \mathfrak{a}_S.
\end{equation}
This gives the first two inequalities in \eqref{Redux2}.
The remaining inequality is simply \eqref{LastStep}, which was proved in Theorem
\ref{EssentialviaSx}.

We now consider \eqref{Redux1}.  If $0<r<1$, there exists a positive integer 
$m(r)$ such that $\bigcup_{j<m(r)}F_j\subset r\Bn$. Then
\begin{align*}
\norm{M_{1_{(r\Bn)^c}}S}_{\mathcal{L}(A_\alpha^p,L_\alpha^p)}
\norm{T^{-1}_\mu}^{-1}_{\mathcal{L}(A_\alpha^p,A_\alpha^p)} &\leq  
\norm{M_{1_{(r\Bn)^c}}ST_\mu}_{\mathcal{L}(A_\alpha^p,L_\alpha^p)}
\\& \leq  \norm{M_{1_{(r\Bn)^c}}
  \left(ST_\mu-\sum_{j=1}^{\infty} 
  M_{1_{F_j}}ST_{1_{G_j}\mu}\right)}_{\mathcal{L}(A_\alpha^p,L_\alpha^p)}
\\&+\norm{M_{1_{(r\Bn)^c}}\sum_{j=1}^{\infty} 
  M_{1_{F_j}}ST_{1_{G_j}\mu}}_{\mathcal{L}(A_\alpha^p,L_\alpha^p)}
\\&\leq \epsilon +\norm{\sum_{j=m(r)}^{\infty}
  M_{1_{F_j}}ST_{1_{G_j}\mu}}_{\mathcal{L}(A_\alpha^p,L_\alpha^p)}
\\&=\epsilon+
  \norm{S_{m(r)}}_{\mathcal{L}(A_\alpha^p,L_\alpha^p)}.
\end{align*}

This string of inequalities easily yields
\begin{equation}
\label{Redux1-1}
\mathfrak{c}_{S}=\varlimsup_{r\to 1}
\norm{M_{1_{(r\Bn)^c}}S}_{\mathcal{L}(A^p_{\alpha},
L^p_{\alpha})}\lesssim \varlimsup_{m\to\infty} 
\norm{S_m}_{\mathcal{L}(A_\alpha^p,L_\alpha^p)}.
\end{equation}
Also, \eqref{important} gives that
\begin{equation}
\label{Redux1-2}
\varlimsup_{m\to\infty} \norm{S_m}_{\mathcal{L}(A_\alpha^p,L_\alpha^p)}\lesssim
\varlimsup_{\abs{z}\to 1}
\norm{M_{1_{D(z,r)}}S} _{\mathcal{L}(A^p_{\alpha},L^p_{\alpha})}\lesssim
\mathfrak{b}_S.
\end{equation}
Combining the trivial inequality $\mathfrak{b}_S\leq\mathfrak{c}_S$ with 
\eqref{Redux1-1} and \eqref{Redux1-2} we obtain \eqref{Redux1}.
\end{proof}

We can now deduce two results from these theorems. First, we give a way to 
compute the essential norm of an operator. 
\begin{cor}
Let $\alpha>-1$ and $1<p<\infty$ and $S\in\mathcal{T}_{p,\alpha}$.  Then
$$
\norm{S}_e\approx \sup_{\norm{f}_{A^p_\alpha}=1}
\varlimsup_{\abs{z}\to 1}\norm{S_z f}_{A^p_\alpha}.
$$
\end{cor}
\begin{proof}
It is easy to see from Lemma \ref{SOTCon} and the compactness of 
$M_{\mathcal{A}}$ that
$$
\sup_{x\in M_{\mathcal{A}}\setminus\Bn}\norm{S_xf}_{A^p_{\alpha}}=
\varlimsup_{\abs{z}\to 1}\norm{ S_z f}_{A^p_{\alpha}}.
$$
But then,
$$
\sup_{x\in M_{\mathcal{A}}\setminus\Bn}\norm{S_x}_{\mathcal{L}(A^p_{\alpha},
A^p_{\alpha})}= \sup_{\norm{f}_{A^p_\alpha}=1}\varlimsup_{\abs{z}\to 1}
\norm{S_z f}_{A^p_\alpha}.
$$
The result then follows from Theorem \ref{EssentialviaSx}.
\end{proof}

The next is the main result of the paper. 

\begin{thm}\label{thm:main}
Let\/ $1<p<\infty$, $\alpha>-1$ and\/ 
$S\in\mathcal{L}(A^p_{\alpha})$.  
Then $S$ is compact 
if and only if\/
$S\in\mathcal{T}_{p,\alpha}$ and $B(S)=0$ on $\partial\Bn$.
\end{thm}
\begin{proof}
If $B(S)=0$ on $\partial\Bn$, Proposition \ref{BerezinVanish} says that $S_x=0$ 
for all $x\in M_{\mathcal{A}}\setminus\Bn$.
So, if $S\in \mathcal{T}_{p,\alpha}$, Theorem \ref{EssentialviaSx} gives that 
$S$ must be compact.

In the other direction, if $S$ is compact then $B(S)=0$ on $\partial\Bn$ by 
\eqref{Vanishing}.
So it only remains to show that $S\in\mathcal{T}_{p,\alpha}$.
Since every compact operator on $A^p_{\alpha}$ can be approximated by finite 
rank operators, it suffices to show that all rank one operators are in 
$\mathcal{T}_{p,\alpha}$. Rank one operators have the form $f\otimes g$, given 
by
$$
(f\otimes g) (h)=\ip{h}{g}_{A^2_\alpha}f,
$$
where $f\in A^p_{\alpha}$, $g\in A^q_{\alpha}$, and $h\in A^p_{\alpha}$.

We can further suppose that $f$ and $g$ are polynomials, since the polynomials 
are dense
in $A^p_\alpha$ and $A^q_\alpha$, respectively (recall that in the vector-valued
case a monomial is simply $z^{n}e$ where $e$ is a constant vector in $\Cd$). 
For a vector-valued function $f$, let $\widetilde{f}$ be the matrix-valued 
function whose diagonal is $f$ and all other entries are zero. That is, the 
$(i,i)$ entry of 
$\widetilde{f}$ is the $i^{th}$ entry of $f$ and all off diagonal entries are 
zero. Also, define $\bf{1}$ to be the vector in $\Cd$ consisting of all $1$'s.
Consider the following computation:
\begin{align*}
\left(T_{\widetilde{f}}(\bold{1}\otimes \bold{1})T_{\widetilde{g^{*}}}\right)h 
&=T_{\widetilde{f}}\left(\ip{P(\widetilde{g^{*}}h)}
    {\bf{1}}_{A_{\alpha}^{2}}\bf{1}\right)
\\&= P\left(\ip{P(\widetilde{g^{*}}h)}{\bf{1}}_{A_{\alpha}^{2}}
    \widetilde{f}\bf{1}\right)
\\&= P\left(\ip{h}{\widetilde{g}\bf{1}}_{A_{\alpha}^{2}}
    \widetilde{f}\bf{1}\right)
\\&=\ip{h}{g}_{A_{\alpha}^{2}}f
\\&=(f\otimes g)(h).
\end{align*}
So, it suffices to show that $\bold{1}\otimes \bold{1}\in 
\mathcal{T}_{p,\alpha}$. Let $W$ be the matrix consisting of all $1$'s, and let 
$\delta_{0}$ be the point mass at $0$. Then:
\begin{align*}
T_{\delta_{0}W}h&=Wh(0)
\\&=\left(\sum_{i=1}^{d}\ip{h(0)}{e_{i}}_{\Cd}\right)\bold{1}
\\&=\int_{\Bn}\ip{h(z)}{\bold{1}}_{\Cd}\bold{1}dv_{\alpha}(z)
\\&= \ip{h}{\bf{1}}_{A_{\alpha}^{2}}\bf{1}
\\&= ({\bf 1}\otimes{\bf 1})(h).
\end{align*}

By Theorem \ref{thm:msw4.7}, \begin{math}T_{\delta_{0}W}\end{math} is a member of $\mathcal{T}_{p,\alpha}$.
\end{proof}
\section{Acknowledgements} 
The first author would like to thank Michael Lacey for supporting him as a
research assistant for the Spring semester of 2014 (NSF DMS grant \#1265570)
and Brett Wick for supporting him as a research
assistant for the Summer semester of 2014 (NSF DMS grant \#0955432).

\end{document}